\documentclass [14pt]{amsart}
\usepackage{amsmath}
\usepackage{amssymb}
\usepackage{amscd}
\usepackage{epsfig}
\usepackage{mathabx}

\newtheorem{lemma}{Lemma}[section]
\newtheorem{theorem}[lemma]{Theorem}

\newtheorem{proposition}[lemma]{Proposition}
\newtheorem{corollary}[lemma]{Corollary}

\theoremstyle{definition}
\newtheorem{definition}[lemma]{Definition}
\newtheorem{example}[lemma]{Example}
\newtheorem*{remark}{Remark}

\DeclareMathOperator{\supp}{supp}

\newcommand{\vL}{\varLambda}

\newcommand{\cB}{\mathcal{B}}

\newcommand{\cE}{\mathcal{E}}

\newcommand{\cL}{\mathcal{L}}
\newcommand{\cM}{\mathcal{M}}

\newcommand{\NN}{\mathbb{N}}

\newcommand{\ZZ}{\mathbb{Z}}

\newcommand{\RR}{\mathbb{R}}
\newcommand{\CC}{\mathbb{C}}

\newcommand{\TT}{\mathbb{T}}

\newcommand{\SSS}{\mathbb{S}}
\newcommand{\XX}{\mathbb{X}}
\newcommand{\WW}{\mathbb{W}}
\newcommand{\YY}{\mathbb{Y}}

\newcommand{\oplam}{\mbox{\Large $\curlywedge$}}

\newcommand{\Cc}{C_{\mathsf{c}}}
\newcommand{\Cu}{C_{\mathsf{u}}}

\newcommand{\WAP}{\mathcal{WAP}}
\newcommand{\SAP}{\mathcal{SAP}}

\newcommand{\Hmm}[1]{\leavevmode{\marginpar{\tiny%
$\hbox to 0mm{\hspace*{-0.5mm}$\leftarrow$\hss}%
\vcenter{\vrule depth 0.1mm height 0.1mm width \the\marginparwidth}%
\hbox to 0mm{\hss$\rightarrow$\hspace*{-0.5mm}}$\\\relax\raggedright
#1}}}

\begin{document}
\title{On Weakly Almost Periodic Measures}

\author{Daniel Lenz}
\address{Mathematisches Institut, Friedrich Schiller Universit\"at Jena, 07743 Jena,
Germany}
\email{daniel.lenz@uni-jena.de}
\urladdr{http://www.analysis-lenz.uni-jena.de}

\author{Nicolae Strungaru}
\address{Department of Mathematical Sciences, MacEwan University \\
10700 " 104 Avenue, Edmonton, AB, T5J 4S2\\
and \\
Department of Mathematics\\
Trent University \\
Peterborough, ON
and \\
Institute of Mathematics ``Simon Stoilow''\\
Bucharest, Romania}
\email{strungarun@macewan.ca}
\urladdr{http://academic.macewan.ca/strungarun/}

\begin{abstract}
We study the diffraction and dynamical properties of translation
bounded weakly almost periodic measures. We prove that the dynamical
hull of a weakly almost periodic measure is a weakly almost periodic
dynamical system with unique minimal component given by the hull of
the strongly almost periodic component of the measure. In particular
the hull is minimal if and only if the measure is strongly almost
periodic and  the hull is  always measurably conjugate to a torus
and has pure point spectrum  with continuous eigenfunctions. As an application we show the
stability of the class of weighted Dirac combs with Meyer set or FLC support and deduce that such measures have
either trivial or large pure point respectively continuous spectrum.  We complement these results by investigating the
Eberlein convolution of two weakly almost periodic measures. Here,
we show that it  is unique and a strongly almost periodic measure.
We conclude by studying the Fourier-Bohr coefficients of weakly
almost periodic measures.
\end{abstract}

\maketitle

\tableofcontents

\section*{Introduction}
Over the past 100 and more years physical
diffraction gave us valuable insight into the atomic structure of
solids. Periodic crystals produce a clear diffraction pattern,
consisting exclusively of bright spots ({\bf Bragg peaks})
positioned on a lattice (the dual lattice of the periods lattice of
the crystal).

While often the atomic structure of the crystal can be reconstructed
from the diffraction pattern, this is not always the case: there
exists different periodic crystals with the same diffraction
pattern. The problem becomes even more complex in the case of
quasicrystals, or more generally, systems with mixed spectrum (see
e.g. \cite{GB} for a recent overview of the inverse problem).

Mathematically, physical diffraction is described as follows:
starting from  the underlying structure, which is modeled by a
discrete set or more generally by a measure,  we construct a measure
$\gamma$, called the {\bf autocorrelation} measure, which encodes
the frequencies of vectors between atoms in the structure. This
measure is positive definite, therefore Fourier transformable with a
positive Fourier transform $\widehat{\gamma}$ \cite{ARMA1,BF,MoSt}.
The positive measure $\widehat{\gamma}$ is called the {\bf
diffraction measure} or {diffraction pattern} and models the
physical diffraction of the original solid. The Lebesgue
decomposition
\[
\widehat{\gamma}=\widehat{\gamma}_{\mathsf{d}}+
\widehat{\gamma}_{\mathsf{c}}
\]
gives us a decomposition of the diffraction into the discrete or
pure point part $\widehat{\gamma}_{\mathsf{d}}$ and the continuous
part $\widehat{\gamma}_{\mathsf{c}}$. (The continuous part can then
be further split into  the  absolutely continuous part  and the
singularly continuous part.) The pure point component
$\widehat{\gamma}_{\mathsf{d}}$ can be written as
\[
\widehat{\gamma}_{\mathsf{d}} = \sum_{\chi \in \cB} \widehat{\gamma} (\{ \chi \}) \delta_\chi \,,
\]
where $\cB$ is the set of \textbf{Bragg peaks}.

As any positive definite translation bounded measure, the
autocorrelation is a weakly almost periodic measure
\cite{ARMA,MoSt}. The Eberlein decomposition
\[
\gamma=\gamma_{\mathsf{s}}+\gamma_{0} \,,
\]
into the strong and null weakly almost periodic part is the Fourier
dual of the above mentioned  Lebesgue decomposition of
$\widehat{\gamma}$  into the discrete and continuous part
\cite{MoSt,ARMA} or \cite[Thm.~2.8]{NS4}. This implies in particular
that $\gamma$ is strongly almost periodic if and only if
$\widehat{\gamma}$ is a pure point measure.

Recently there is also increased interest in diffraction of systems
which are described by more general structures than measures
\cite{BLPS,LM,ter,teba}. In particular, \cite{BLPS,ter,teba} deal
with tempered distributions. Also, in these cases we end up with
weakly almost periodic distributions and there exists also an Eberlein
decomposition of tempered distributions which is dual to the
Lebesgue decomposition \cite{ST}.

Almost periodicity has long been known to play a role in the study
of  pure point diffraction, see e.g. the paper \cite{Sol} of
Solomyak for an early explicit investigation. The possible
connections between almost periodicity properties of the original
point set or measure and pure point diffraction were then the topic
of a survey article of Lagarias \cite{LAG}, which contains many
questions and conjectures on this topic. The actual statement on the
equivalence of pure point diffraction and strong almost periodicity
of the autocorrelation measure discussed above was first given
explicitly in the work of Baake-Moody \cite{BM}. The main thrust of
their work concerns the case of measures with Meyer set support. In
that case, strong almost periodicity of the autocorrelation can even
be replaced by the stronger notion of norm almost periodicity
\cite{BM}. Connections of strong almost periodicity of measures and
 pure point diffraction were  then also studied by Moody-Strungaru \cite{MS} and Gou\'er\'e
\cite{Gouere-1} (see also  Favarov \cite{FAV} and  Kellendonk-Lenz
\cite{KL} for related material discussing strongly almost periodic
point sets). Also, the connection between almost periodicity and the
cut and project formalism thoroughly investigated by Baake-Moody
\cite{BM}, Richard \cite{CR}, Lenz-Richard \cite{LR} and Strungaru
\cite{NS11}. As shown recently in Lenz \cite{Len} it is even
possible to extend parts of the above theory from dynamical systems
based on measures to general dynamical systems allowing for a
metric.

All these results emphasize that the class of (weakly) almost
periodic measures is important for diffraction theory. It is our
goal in this paper to systematically study the long range order
properties of this class via their dynamical system and the
diffraction measure.

As far as methods and tools go the paper builds up on two related
theories. On the one hand it can be seen as a continuation of the
work of Eberlein \cite{EBE,EBE2,EBE3}  on weakly almost periodic
functions and measures in the context of pure point diffraction. On
the other hand it can be seen as  an application of the theory of
weakly almost periodic systems to a specific class of measure
dynamical systems. This theory was  started by Ellis and Nerurkar
\cite{EN}  and has  later been extended  in various works, see e.g.
the  recent work of Akin-Glasner \cite{AG}, where various further
references can be found.

The paper is organized as follows: we start by reviewing  basic
concepts in harmonic analysis on locally compact abelian groups in
Section \ref{sec:prel}. There, we also introduce one of the main
players of our investigation viz the autocorrelation of a measure.
We will be interested in the autocorrelation of weakly almost
periodic measures and much of our study is built on dynamical
systems. The corresponding background dynamical systems, measure
dynamical systems and  on almost periodic measures and corresponding
topologies on the set of all translation bounded measures is
discussed in Section \ref{sec:WAP} This material is essentially all
known. We discuss it at some length in order to keep the present
paper accessible to a wide audience.

Our actual investigation starts in the subsequent Section
\ref{sec:dyn-syst}. There,  we present a thorough analysis  of the
hull of a weakly almost periodic measure. In particular, we show
 in Section \ref{sec:hullisweaklyalmostperiodicsystem} that the hull is
a weakly almost periodic dynamical system. In the subsequent Section
\ref{sec:stability} we discuss how the classes of weakly,
null-weakly and strongly almost periodic measures are stable under
taking the hull. We show that the hull of a weakly almost periodic measure has  a
unique minimal component in Section \ref{sec:WAP hull}. In this
context, a main result of the paper, Theorem \ref{thm:Hull WAP}, is
proven. This result shows that for every weakly almost periodic
measure $\mu$, its strongly almost periodic component
$\mu_{\mathsf{s}}$ belongs to the hull $\XX(\mu)$. Most of the
subsequent considerations rely on this result. One implication is
that  the hull generated by $\mu_{\mathsf{s}}$ is the unique minimal
component. In fact, based on this result we  can discuss the
structure of the hull of a weakly almost periodic system in quite
some detail in  the subsequent Section \ref{sec:Hull structure}.
There, we show that the projections $P_{\mathsf{s}}, P_{0}$ onto the
strongly and the null-weakly almost periodic part are continuous
$G$-mappings with range contained inside $\XX(\mu_{\mathsf{s}})$
respectively $\XX(\mu_0)$.

In Section \ref{sec:dynam spectrum} we show that the dynamical
system of a weakly almost periodic measure $\mu$ has pure point
dynamical and diffraction spectra. Here, we also  provide a closed
formula for the eigenfunctions of the system, which are shown to be
continuous, and for the diffraction measure. When combining this
with Theorem \ref{thm:bragg peaks have continuous eigenfunctions}
from Section \ref{sec:stability} we infer as a remarkable feature of
our setting that  the eigenfunctions relevant for diffraction are
just given by the Fourier coefficients of the measure in question.

The preceding  results on weakly almost periodic measures have an
application to arbitrary measure dynamical systems. The reason is
that each measure dynamical system gives rise to an autocorrelation
which is weakly almost periodic. It turns out that the diffraction
of the dynamical system arising as the hull of the  autocorrelation
is just the pure point part of the diffraction of the
original system. This is discussed in Section \ref{sec:hull-autororrelation}.

Our results have various application to Delone sets. This is
discussed in Section \ref{sec:Application}. There, we  first relate
in Section \ref{sec:Delone}, for weakly almost periodic weighted
Dirac combs $\mu$ with FLC support, the support of
$\mu_{\mathsf{s}}$ and of $\mu_0$ to the support of $\mu$. We then
use our results to give rather direct alternative proofs to recent
results on the support of the Eberlein decomposition for weighted
Dirac combs. This is discussed in Section \ref{sect: ebe decomp and
supp}.

We complete the paper by looking at the Eberlein convolution of
weakly almost periodic measures in Section \ref{sec:ebe-conv}. We
show that given any two weakly almost periodic measures they have an
unique Eberlein convolution, which exists with respect to any van
Hove sequence and is strongly almost periodic. As a consequence we
get (an alternative proof)  that any weakly almost periodic measure
has pure point diffraction and hence pure point dynamical spectrum.

By combining weakly almost periodic dynamical systems and aperiodic
order in  our paper we emphasize that there is a deep connection
between these concepts, which could be of interest to mathematicians
from both areas. In order to keep the paper  accessible for both
communities we have taken special care to discuss basic concepts at
sufficient length.

\section{General background and notation}\label{sec:prel}
For this entire paper $G$ denotes a locally compact Abelian group
(LCAG). We will denote by $\Cu(G)$ the space of uniformly continuous
and bounded functions on $G$, and by $\Cc(G)$ the subspace of
$\Cu(G)$ consisting of functions with compact support.

For a compact set $K \subset G$ we define
\[
C(G:K):= \{ f \in \Cc(G) | \supp(f) \subset K \} \,.
\]
The space $( C(G:K), \| \, \|_\infty)$ is a Banach space. On
$\Cc(G)$ we define a locally convex topology as the inductive limit
topology defined by the embeddings $C(G:K) \hookrightarrow \Cc(G)$,
for $K\subset G$ compact.

For a function $f$ on $G$ we define the functions $f^\dagger$ and
$\widetilde{f}$ on $G$ by
\[
f^\dagger(x)=f(-x) \, \mbox{ and } \,
\widetilde{f}(x)=\overline{f(-x) }.
\]

Similarly, for a measure $\mu$ we define a new measure
$\widetilde{\mu}$ by
\[
\widetilde{\mu}(f) := \overline{ \mu( \widetilde{f}) } \,.
\]

By Riesz-Markov theorem, a measure $\mu$ on $G$ can be viewed as a
linear functional on $\Cc(G)$, which is continuous with respect to
the inductive topology on $\Cc(G)$, see \cite{Ped,ReiterSte} or
\cite[Appendinx]{CRS2} for details. We will refer to the weak-*
topology of this duality as the {\bf vague} topology for measures.

\medskip

Any measure $\mu$ gives rise to an unique positive measure $\left| \mu \right|$, called the {\bf variation measure of $\mu$}, satisfying
\[
\left| \mu \right|(f) = \sup \{ \left| \mu(g) \right| : g \in \Cc(G:\RR), \mbox{ with } |g| \leq f \} \,,
\]
for any non-negative $f \in \Cc(G)$.  We can define the convolution
between $c \in \Cc(G)$ and $g \in \Cu(G)$ as
\[
c*g(x) =\int_G c(x-t) g(t) d t \,.
\]
Then $c*g \in \Cu(G)$. Similarly, the convolution between a function
$c \in \Cc(G)$ and a measure $\mu$ is the function
\[
c*\mu (x) =\int_G c(x-t) d \mu(t) \,.
\]
We will restrict  attention to a class of measures which are
"equi-bounded" on $G$:

\begin{definition} A measure $\mu$ on $G$ is called {\bf translation bounded} if for each compact $K \subset G$ we have
\begin{equation}\label{EQtb}
\| \mu \|_K := \sup_{x \in G} \left| \mu \right|(x+K) < \infty
\end{equation}
where $ \left| \mu \right|$ denotes the variation of $\mu$. The
space of all translation bounded measures on $G$ is denoted by
$\cM^\infty(G)$.
\end{definition}

To check that a measure is translation bounded, it is sufficient to
check that (\ref{EQtb}) holds for one fixed compact set $K$ with
non-empty interior \cite{BM}. Also, the translation boundedness
property of measures can be understood via convolutions with
compactly supported continuous functions. Indeed, by
\cite[Thm.1.1]{ARMA1} a measure $\mu$ is translation bounded if and
only if for all $c \in \Cc (G)$ we have $c*\mu \in \Cu(G)$.  The
convolution of function can be extended to measures the following
way: We say that the measures $\mu$ and $\nu$  are
\textbf{convolvable}, if for all $f \in \Cc(G)$ the function $(x,y)
\to |f(x+y)|$ is in $L^1(|\mu| \times |\nu|)$. In this case, we
define the convolution $\mu*\nu$ by
\[
\mu*\nu (f) = \int_G \int_G f(x+y) d \mu(x) d \nu (y).
\]

In general for translation bounded measures we can try to also
define an averaged convolution (Eberlein convolution) of the
measures. This point will be taken up and discussed extensively
later on in the last section of the article. Here, we introduce next
what makes a good averaging set.

\begin{definition}
A sequence $(A_n)$  of compact subsets of $G$  called a {\bf van
Hove sequence} if for all compact sets $K \subset G$ we have
\[ \lim_{n \to \infty} \frac{\theta_G(\partial^K(
A_n)) } {\theta_G(A_n) } = 0\,, \] where the {\bf $K$-boundary} is
defined by:
\[\partial^K(A_n)=\overline{((A_n+K)\backslash A_n)} \cup((\overline{G \backslash A_n} - K) \cap B_n) \,.\]

A sequence $(A_n)$ of compact subsets of $G$ is called a {\bf
F\"{o}lner sequence} if for all $x \in G$ we have
\[ \lim_{n \to \infty} \frac{\theta_G(A_n \bigtriangleup (x+A_n)) } {\theta_G(A_n) } =
0\,. \]
\end{definition}

It is easy to see that any van Hove sequence is a F\"{o}lner
sequence.

\smallskip

We are now ready to introduce the notion of autocorrelation measure.

\begin{definition}
Let $\mu \in \cM^\infty(G)$ be a measure and let $(A_n)$ be a van
Hove sequence. We say that $\mu$ has the  autocorrelation $\gamma$
with respect to $(A_n)$ if the following vague limit exists:
\[
\gamma:= \lim_n \frac{1}{\theta_G(A_n)} \mu|_{A_n} *\widetilde{\mu|_{A_n}} \,,
\]
where $\mu|_{A_n}$ denotes the restriction of the measure $\mu$ to
the set $A_n$. (Here, the convolution makes sense as  the measures
$\mu|_{A_n}$ and $\widetilde{\mu|_{A_n}}$ are finite due to the
compactness of the $A_n$.)
\end{definition}

Note that existence of the limit in the definition is rather a
matter of convention. Indeed, the limit will always exist for a
suitable subsequence of $(A_n)$.  If the limit exists for all van
Hove sequences $(A_n)$, it must be the same for all. In this case we
will say that $\mu$ {\bf has an unique autocorrelation}.

\smallskip

The autocorrelation has an extra positivity property which we
introduce next:
\begin{definition} A function $f : G\longrightarrow \CC$ is called {\bf positive
definite} if for all $N\in \NN$ and $x_1,\ldots, x_N\in G$, the
matrix $(f (x_k - x_l))_{k,l=1,\ldots, N}$ is positive Hermitian. A
measure $\mu$ on $G$  is called {\bf positive definite} if for all
$f \in  \Cc(G)$ we have
\[ \mu  (f*\widetilde{f}) \, \geq \, 0 \,.\]
\end{definition}

It is well known (see e.g. \cite{BF}) that a measure $\mu$ is
positive definite if and only if  for all $f \in \Cc(G)$ the
function $\mu * (f* \widetilde{f})$ is positive definite.

It is also easy to see that for any finite measure $\nu$, the
measure $\nu*\widetilde{\nu}$ is positive definite, and that vague
limits of positive definite measures are positive definite. Thus, we
immediately obtain the following corollary.

\begin{corollary} Let $\gamma$ be the autocorrelation of $\mu$ with
respect to the van Hove sequence $(A_n)$. Then, $\gamma$ is positive
definite.
\end{corollary}

\section{Measure dynamical systems and weakly almost periodic measures}\label{sec:WAP}
The use of dynamical systems in the study of long range order has a
long history, see e.g. the survey  \cite{Rad} for an early
account. As in statistical mechanics, the basic idea is to look at
ensembles rather than at individual manifestations. More
specifically, the idea is  to construct out of a given point
set $\Lambda \subset G$, the collection $\XX(\Lambda)$ of all point
sets which locally look like $\Lambda$. This collection becomes
compact with respect to the local rubber topology, and comes with a
natural $G$-action, therefore it becomes a topological dynamical
system. The diffraction of $\Lambda$ can then be connected to the
dynamical spectrum of $\XX(\Lambda)$ via the Dworkin argument
\cite{BL,Gouere-1,LMS-1,Dwo}, and all these ideas can be generalized
to measures \cite{BL,LS}.  We recommend \cite{BL3} for a recent
review of this.

\smallskip

\subsection{Measure dynamical systems}
We will be interested in dynamical systems coming from translation
bounded measures.

\smallskip

Whenever $\XX$ is a compact space equipped with a continuous action
$T :G\times \XX \longrightarrow \XX$  of the group $G$, we call
$(\XX,T)$ a dynamical system. Such a system is called
\textbf{minimal} if the \textbf{orbit} $\{T^t x: t\in G\}$ of $x\in
\XX$ is dense in $\XX$ for every $x\in \XX$. If there exists one
dense orbit $\{T^t x : t\in G\}$ then the system is called
\textbf{transitive}.
 A system  is called \textbf{uniquely ergodic} if there exists only
one $G$-invariant probability measure on $\XX$. A special focus of
ours will be on continuous eigenfunctions. Here, an $f\in C (\XX)$
with $f\neq 0$ is called an \textbf{eigenfunction} to $\chi
\in\widehat{G}$ if
$$f(T^t x) = \chi(t) f(x)$$
holds for all $t\in G$ and $x\in \XX$. Such a $\chi$ is called a
\textbf{continuous eigenvalue}.

We will be interested in special transitive dynamical systems coming
from measures and more specifically coming from weakly almost
periodic measures. In this context we define for a translation
bounded measure $\mu$ on $G$ the measure $T^t \mu$ by
\[
T^t \mu :=\delta_t *\mu,
\] where $\delta_t$ denotes the unique point mass at $t\in G$.

It is not hard to see that the translation action
\[ G\times \cM^\infty (G) \longrightarrow \cM^\infty (G), (t,\mu)\mapsto T^t\mu, \]
 is continuous when restricted to compact $G$-invariant
subsets of $\cM^\infty (G)$. So any such compact $G$-invariant $\XX$
is a dynamical system $(\XX,T)$. We refer to such a dynamical system
as \textbf{measure dynamical system}.

In particular, any $\mu\in\cM^\infty (G)$ gives rise to a dynamical
system $\XX (\mu)$ defined as
\[
\XX (\mu):=\overline{ \{T^t \mu \mid  t\in G\}},\] where the closure
is taken in the vague topology. This dynamical system is called the
\textbf{hull} of $\mu$. We recommend \cite{BL} for more details.

\smallskip

A crucial feature of any measure dynamical system $(\XX,T)$ is that
there is a linear  map $\phi$ from $\Cc(G)$ to the set $\Cc(\XX)$ of
continuous functions on $\XX$.  More specifically,  each $c \in
\Cc(G)$ defines a function $\phi_c: \cM^\infty \to \CC$ via
\[
\phi_c(\nu)=c*\nu(0) = \int_G c(-x) d \nu(x) \,.
\]
These functions are continuous \cite{BL}, and in fact the vague
topology on $\cM^\infty(G)$ is the weakest topology which makes
these functions continuous \cite{BL}. When dealing with a given
measure dynamical system we can (and will tacitly) restrict these
functions to the dynamical system. Accordingly, these functions will
appear often in our computations. We will also make often use of
\[
\phi_c(T^t \nu) =c*\nu(t).
\]

\subsection{Weakly almost periodic functions and
measures} In this section we review the basic definitions and
properties of weakly almost periodic functions and measures.

\bigskip

We start by recalling the notions of almost periodicity we will use.
Here, for a function $f$ on $G$ and $t\in G$ we define the translate
$T^t f$ of $f$ by $t$ to be the function with
\[
T^t f (x) = f(-t + x)\] for all $x\in G$.

\begin{definition}
A function $f \in \Cu(G)$  is called {\bf Bohr-almost periodic} if
the set $\{ T^tf | t \in G \}$ is precompact in $(\Cu(G), \| \,
\|_\infty)$. The set of all Bohr-almost periodic function on $G$ is
denoted by $SAP (G)$.

A function  $f \in \Cu(G)$ is called {\bf weakly almost periodic} if
the set $\{ T^tf | x \in G \}$ is precompact in the weak topology of
the Banach space $(\Cu(G), \| \, \|_\infty )$. The  space of all
weakly almost periodic functions is denoted by $WAP (G)$.
\end{definition}

We summarize now the basic structural stability  properties of the
spaces $WAP(G)$ as proven in \cite{EBE}  and of $SAP (G)$ as
discussed recently in \cite{MoSt}.

\begin{theorem} \label{thm-structure-wap} \cite[Thm.~11.2,Thm.~12.1]{EBE}, \cite{MoSt}.
\begin{itemize}
\item[(a)] $SAP(G)$ and $WAP(G)$ are closed subspaces of $(\Cu(G), \| \, \|_\infty)$.
\item[(b)] $SAP(G) \subset WAP(G)$.
\item[(c)] $SAP(G)$ and $ WAP(G)$ are closed under translation, reflection and complex conjugation.
\item[(d)] $SAP(G)$ and $WAP(G)$ are closed under multiplication.
\item[(e)]  If $f \in SAP(G)$ and $p \geq 0$ then $|f|^p \in SAP(G)$.
\item[(f)]  If $f \in WAP(G)$ and $p \geq 0$ then $|f|^p \in WAP(G)$.
\end{itemize}
\end{theorem}

For our further investigation we will also
need the following fundamental fact.

\begin{theorem}\cite{EBE}  If $f \in \Cu(G)$ is a positive definite function on $G$ then $f$
belongs to $ WAP(G)$.
\end{theorem}

\smallskip

Next we recall the notion of amenability for continuous functions.

\begin{definition} Let $ (A_n)$ be a  F\"olner sequence. A function $f \in \Cu(G)$  is called {\bf amenable} if there exists a number $M(f)$, called the {\bf mean of $f$} such that
\[\lim_n \frac{1}{\theta_G(A_n)} \int_{x+A_n} f(t) dt= M(f) \,,\]
uniformly in $x\in G$.
\end{definition}

\begin{remark} \cite{EBE} The definition of amenability and the mean are independent of the choice of the F\"olner sequence.
\end{remark}

The basic properties of the mean $M : WAP(G) \to \CC$ are summarised below.

\begin{theorem}\label{properties of mean} \cite[Thm.~14.1]{EBE}
Any weakly almost periodic function $f$ is amenable. If $f,g \in
WAP(G), C_1,C_2 \in \CC$ and $t \in G$ then we have:
\begin{itemize}
\item[(a)] $M(C_1f+C_2g)=C_1M(f)+C_2M(g)$.
\item[(b)] $M(1)=1$.
\item[(c)] If $f \geq 0$ then $M(f) \geq 0$.
\item[(d)] $|M(f) | \leq M( |f |) \leq \| f \|_\infty $.
\item[(e)]  $M(\overline{f})=\overline{M(f)}$.
\item[(f)] $M(T^tf)=M(f)$.
\item[(g)]  $M(f^\dagger)=M(f)$.
\item[(h)]
\[\left| M(fg) \right|^2 \leq \left(M(|f|^2) \right) \left(M(|g|^2) \right) \,.\]
\end{itemize}
Moreover, if $L : WAP(G) \to \CC$ is a function which satisfies (a),
(b), (c), (d) and (f) then  $L(f)=M(f)$ for all $f \in \WAP(G)$.
\end{theorem}

As a  consequence of Theorem \ref{thm-structure-wap} and Theorem
\ref{properties of mean} we note the following.

\begin{corollary} \cite{EBE} Let $f \in WAP(G)$. Then $M(|f|)=0$ if an only if $M(|f|^2)=0$.
\end{corollary}
\smallskip

We are now ready to introduce the notion of null-weakly almost
periodicity.

\begin{definition} A function $f$ is called {\bf null weakly almost periodic} if $f \in WAP(G)$ and $M(|f|)=0$.
We will denote the space of null weakly almost periodic functions by $WAP_0(G)$.
\end{definition}

The relevance of the space of null-weakly almost periodic functions
comes from  the following result.

\begin{theorem} \cite{EBE3} $ WAP(G) = SAP(G) \oplus WAP_0(G).$
\end{theorem}

The theorem says that any weakly almost periodic function  $f$ can
uniquely be written as the sum
\[
f=f_{\mathsf{s}}+f_0 \,,
\]
of a strongly almost periodic function $f_{\mathsf{s}}$ and a weakly almost periodic function $f_0$. We will refer to this decomposition as the {\bf Eberlein decomposition} of weakly almost periodic functions.

\smallskip
Next we see the Eberlein convolution of functions:

\begin{definition}
If $f, g \in WAP(G)$ we can define a new function $f \circledast g$
via
\[
f \circledast g (x) = M_t( f(x-t)g(t)),
\]
where $M_t$ denotes the mean with respect to the variable $t$. The
function $f\circledast g$ is called the {\bf Eberlein convolution}
of $f$ and $g$. (For a fixed $x\in G$ we have  $f(x- \cdot)g(\cdot)
\in WAP(G)$ and therefore the mean exists.)
\end{definition}

\begin{theorem} \cite[Thm.~15.1]{EBE} Given $f, g \in WAP(G)$ we have $f \circledast g \in SAP(G)$ and
\[
f \circledast g (x) = M_t( f(t)g(x-t)) \,.
\]

\end{theorem}

\smallskip

The above concepts can be extended to measures via convolutions. This has been carried out  by deLamadrid and Argabright \cite{ARMA} (see also \cite{MoSt}):

\begin{definition}
A translation bounded measure $\mu$  is called {\bf strongly,
weakly} respectively {\bf null weakly almost periodic} if for all $c
\in \Cc(G)$ the function $c*\mu$ is Bohr, weakly respectively null
weakly almost periodic.  We denote the spaces of weakly, strongly
and null weakly almost periodic measures by $\WAP(G), \SAP(G)$
respectively $\WAP_0(G)$.
\end{definition}

Exactly as for functions, we can define in a simple way the mean of
a weakly almost periodic measure.

\begin{proposition}\label{mean lemma}\cite{ARMA,MoSt}
Let $\mu$ be a weakly almost periodic measure. Then, there exists a
constant $M(\mu)$ such that, for all $c \in \Cc(G)$ we have
\[
M(c * \mu) = M(\mu) \int_G c(t) dt \,.
\]
\end{proposition}

\begin{definition}
Given a measure $\mu \in \WAP(G)$, we call the number $M(\mu)$ from
Proposition \ref{mean lemma}, {\bf the mean} of $\mu$.
\end{definition}

The mean of a measure can be calculated by a formula which is
similar to the formula for  functions.

\begin{lemma}\cite{MoSt} Let $\mu \in \WAP(G)$ and $(A_n)$ a van-Hove sequence on $G$.  Then
\[
M(\mu) =\lim_n \frac{\mu(A_n)}{\theta_G(A_n)} \,.
\]
\end{lemma}

The Eberlein decomposition can then be extended to measures:

\begin{theorem} \cite{ARMA,MoSt}$ \WAP(G) = \SAP(G) \oplus \WAP_0(G)$.
\end{theorem}

As for functions, we will refer to this decomposition as the
\textbf{Eberlein decomposition}.

\smallskip

Let us emphasize here that the space of null weakly almost periodic
measures is more enigmatic than the space of null weakly almost
periodic functions:  Consider a weakly almost periodic measure
$\mu$. Then, by definition,  it is null weakly almost periodic if
and only if for all $c \in \Cc(G)$ we have $M(|c*\mu|)=0$. However,
this is not equivalent to $M(|\mu|)=0$ (even if the variation
measure $|\mu|$ is weakly almost periodic). Indeed,  as proven  in
\cite{NS1,NS11}, for weakly almost periodic measures with uniformly
discrete support we have the equivalence $\mu \in \WAP_0(G)$ if and
only if $M(|\mu|)=0$, but in general we only have the implication $
M(|\mu|)=0 \Rightarrow \mu \in \WAP_0(G)$. The measure
\[ \mu =\sum_{ m \in \ZZ \backslash \{ 0 \} } \delta_{m+\frac{1}{m}}-\delta_m \,,\]
is a null weakly almost periodic measure, but $M(|\mu|) =2$, see \cite{NS1}.

\smallskip

Next we introduce the Fourier Bohr coefficients of a weakly almost periodic measure. Whenever $f$ is strongly
almost periodic function and $\mu$ is a weakly almost periodic
measure we have $f\mu \in \WAP(G)$\cite[Thm.~6.2]{ARMA}. We can thus define.

\begin{definition}
The {\bf Fourier Bohr} coefficients of a weakly almost periodic
measure are defined for each  character $\chi \in \widehat{G}$ as
\[
c_\chi(\mu) := M(\bar{\chi} \mu) \,.
\]
Similarly, if $f \in WAP(G)$ we define the {\bf Fourier Bohr}
coefficients of $f$ as
\[
c_\chi(f) := M(\bar{\chi} f) \,.
\]
It is easy to see that $c_\chi(f)=c_\chi(f \theta_G)$.
\end{definition}

As proven in \cite[Thm.~8.1]{ARMA}, the null weakly almost periodic
measures/functions are exactly the weakly almost periodic
measures/functions with vanishing Fourier -Bohr coefficients. It
follows from the uniqueness of the Eberlein decomposition that a
strongly almost periodic measure/function is uniquely determined by
its Fourier Bohr coefficients.

The Fourier Bohr coefficients of the Eberlein convolution are
exactly the products of the Fourier Bohr coefficients of the two
functions. This result  is proven in \cite[Lemma 1.15]{EBE} in the
particular case $g=\widetilde{f}$ and in \cite{ARMA} in the case of
a convolution between a strong almost periodic function and a weakly
almost periodic measure. For completeness reasons we  include a
proof (which is the standard one) for the general case.

\begin{theorem}
Let $f,g \in \WAP(G)$ and $\chi \in \widehat{G}$. Then
\[
c_\chi(f\circledast g)=c_\chi(f)c_\chi(g) \,.
\]
\end{theorem}
\begin{proof} A direct computation gives
\begin{eqnarray*}
c_\chi(f\circledast g)&=&M_x(\bar{\chi}(x) f \circledast g(x))= M_x\left(\bar{\chi}(x) M_t( f(x-t)g(t)) \right) \\
&=& M_x\left(M_t( \bar{\chi}(x) f(x-t)g(t))\right) \,.
\end{eqnarray*}
Now, by the Fubini Theorem for the mean \cite[Thm.14.2]{EBE}, we
have
\begin{eqnarray*}
c_\chi(f\circledast g)&=&M_t\left(M_x( \bar{\chi}(x) f(x-t)g(t))\right)=M_t\left(\bar{\chi}(t) g(t) M_x( \bar{\chi}(x-t) f(x-t))\right)\\
&=&M_t\left(\bar{\chi}(t) g(t) a_\chi(f) \right) \,.
\end{eqnarray*}
As $a_\chi(f)$ is a constant, the claim follows.
\end{proof}

\subsection{Product and weak  topology on $\cM^\infty (G)$ and $\WAP(G)$}
In this section we provide another perspective on the weak and
strong almost periodic measures by considering two further
topologies on $\cM^\infty (G)$.

\medskip

Consider the natural embedding
\[
\cM^\infty(G) \hookrightarrow [\Cu(G)]^{\Cc(G)}  \quad;\quad \mu \to
\{ c*\mu  \}_{c \in \Cc(G)} \,.
\]
Via this embedding the product topology on $[\Cu(G)]^{\Cc(G)}$
induces a topology on $\cM^\infty(G)$. This topology is called the
{\bf product topology for measures}.  Let us note that the product
topology is a locally convex topology defined by the family of
seminorms $\{ \| \quad \|_c \}_{c \in \Cc(G)}$ given by
\[
\| \mu \|_c := \| c*\mu \|_{\infty} \,.
\]
Thus, it also allows for a corresponding weak topology.  We will
refer to this weak topology as {\bf the weak topology for measures}.

\begin{proposition}\label{Prop:vague weaker than weak} The vague topology is weaker than
the weak topology for measures
\end{proposition}
\begin{proof} This follows since for each $c \in \Cc(G)$ the mapping
\[
\mu \to \mu(c) \,,
\]
is a linear functional which is continuous with respect to the
product topology.
\end{proof}

\begin{remark}It is not known if the image of $M^\infty (G)$ under the above embedding
  is closed in $ [\Cu(G)]^{\Cc(G)}$, but it is bounded closed, and hence quasi-complete (See \cite[Thm.~2.4, Cor.~2.1]{ARMA}).
\end{remark}

Here comes a characterization of (weak) almost periodicity of
measures via the previously introduced topologies.

\begin{theorem} \cite{ARMA} Let $\mu \in \cM^\infty(G)$. Then
\begin{itemize}
  \item [(a)] A measure $\mu$ belongs to $\SAP(G)$ if and only if $\{ T_t \mu | t \in G\}$ is precompact in the product topology.
  \item [(b)] A measure $\mu$ belongs to $ \WAP(G)$ if and only if $\{ T_t \mu | t \in G\}$ is precompact in the weak topology of the product topology.
\end{itemize}
\end{theorem}

We complete the section by showing the continuity of the functions
$c_\chi$ on $\WAP(G)$. We start with a standard lemma from topology.

\begin{lemma}\cite{LOO}\label{compact}
Let  $\tau_1$ and $ \tau_2$ be topologies on $X$ such that $\tau_1$
is stronger than $\tau_2$ and $\tau_2$ is a Hausdorff topology.  If
$A \subset X$ is a compact set with respect to $\tau_1$ then $A$ is
compact with respect to $\tau_2$ and the two topologies coincide on
$A$.
\end{lemma}

\begin{theorem} \label{thm: cont FB coeff} Let $\chi \in \widehat{G}$.
Then, the Fourier coefficient  $c_\chi : \WAP (G)\longrightarrow
\CC$ is weakly continuous on $\WAP(G)$.  In particular, if $K$ is
any weakly compact subset of $\WAP(G)$, then $c_\chi$ is vaguely
continuous on $K$.
\end{theorem}
\begin{proof}
 First, let us observe that we have
\[
\left| c_\chi(f) \right| \leq \| f\|_\infty \,,
\]
This shows that $c_\chi$ is continuous on $(WAP(G), \|
\,\|_\infty)$. Then, by the Hahn Banach Theorem, we can extend this
to a continuous functional $f_\chi$ on $(\Cu(G), \| \, \|_\infty)$.
Furthermore, by the definition of the product topology, for each $c
\in \Cc(G)$ the mapping
$$(\cM^\infty(G),\mbox{product topology}) \to (\Cu (G), \|\,\|_\infty), \mu \to
c*\mu$$ is continuous from $\cM^\infty$.  It follows that this
mapping is weakly continuous, see e.g. \cite{MoSt}.  Therefore, the
composition $\mu \to f_\chi(c*\mu)$ is a weakly continuous mapping
on $\cM^\infty$, and hence its restriction $\mu \to c_\chi(c*\mu)$
is weakly continuous on $\WAP(G)$.  Finally, let us pick some $c \in
\Cc(G)$ with $\widehat{c}(\chi) \neq 0$. Then, as
\[
c_\chi(c*\mu) =\widehat{c}(\chi) c_\chi(\mu),
\]
we get the weak continuity of $c_\chi$ on $\WAP(G)$.

The last statement follows as  by Lemma \ref{compact} and
Proposition \ref{Prop:vague weaker than weak}, the vague and weak
topologies coincide on $K$.
\end{proof}

\begin{remark}
When we extended $c_\chi$ to a continuous linear functional on
$\Cu(G)$ in order to get the weak continuity of $c_\chi$ we have
actually seen a very particular case of the following general
phenomena, which is an immediate consequence of the Hahn-Banach
Theorem: If $(E, \| \, \|)$ is a Banach space, and $F$ is a closed
subspace of $E$, then the weak topology of $(F, \| \, \|)$ is the
same as the inherited weak topology on $F$ as a subspace of $E$.
\end{remark}

\section{The hull of a weakly almost periodic measure}\label{sec:dyn-syst}
In this section we present a rather thorough study of the hull of a
weakly almost periodic measure.

\subsection{The hull is a weakly almost periodic
system}\label{sec:hullisweaklyalmostperiodicsystem} In this section
we provide an important structural insight concerning the hull of a
weakly almost periodic measure: This hull is a weakly almost
periodic dynamical system. This makes a wealth of results available
to the study of weakly almost periodic measures.

\smallskip

Whenever $(\XX,T)$ is a dynamical system a  function $f \in C(\XX)$
is called \textbf{weakly almost periodic} if the set $\{ f\circ T^t
: t\in G\}$ is relatively compact in the weak topology of
$(C(\XX),\|\cdot\|_\infty)$.  We denote by $WAP(\XX)$ the subset of
$C(\XX)$ consisting of weakly almost periodic functions. A dynamical
system $(\XX,T)$ is called \textbf{weakly almost periodic} if every
$f\in C (\XX)$ is an weakly almost periodic function, or
equivalently if $C(\XX)=WAP(\XX)$.

In order to prove the main result of this section we next carry out
a study of $WAP (\XX)$ for general dynamical systems. As is will
appear in many proofs, we define  \textbf{weak the hull of $f$} by
\[
\WW(f):= \overline{\{ f\circ T^t : t\in G \} }
\]
the closure being in the weak topology of $C(\XX)$.

\begin{lemma}\label{lem: weak convergence compact}(\cite[Thm.~1.3]{EBE}, \cite[Prop.~1.4.8]{MoSt}) If $\Omega$ is a
compact space, and $f_n, f \in C(\Omega)$, then $f_n \to f$ weakly
if and only if $\| f_n \|_\infty$ is bounded and $f_n \to f$
pointwise.
\end{lemma}

Next, we will often make use of the following classical result,
which allows us to use sequences instead of nets to prove weakly
almost periodicity.

\begin{lemma}[Grothendieck]\cite[Thm.~1.43 (1)]{GLAS}
\label{cor: WAP by sequences} Let $f \in C(\XX)$. Then $f \in
WAP(\XX)$ if and only if for each sequence  $(t_n)$ in  $G$ we can
find a subsequence $(t_{k_n})$ such that $f \circ T^{t_{k_n}}$ is
weakly convergent to some $g \in C(\XX)$.
\end{lemma}

\smallskip

We are now ready to prove a few useful results about the space
$WAP(\XX)$. Note that the proofs are almost identical to the ones
from \cite{EBE,MoSt}, but since those in the cited papers dealt only
with $WAP(G)$, we are in a completely different situation.
Nevertheless the basic ideas and tools remain the same.

\begin{proposition}\label{prop: WAP is algebra}
$WAP(\XX)$ is a subalgebra of $C(\XX)$ which contains the constant
function $1$.
\end{proposition}
\begin{proof} Let $f , g \in WAP(\XX)$, and let $(t_n)$ be any sequence in $G$.
By  $f \in WAP(\XX)$ there exists by Corollary \ref{cor: WAP by
sequences}  a subsequence $(t_{k_n})$ such that $f \circ
T^{t_{k_n}}$ is weakly convergent to some $f_1 \in C(\XX)$.
Similarly, by  $g \in WAP(\XX)$, there exists again  by Corollary
\ref{cor: WAP by sequences} a subsequence $(t_{l_n})$ of $(t_{k_n})$
such that $g \circ T^{t_{l_n}}$ is weakly convergent to some $g_1
\in C(\XX)$.

As $f \circ T^{t_{l_n}}$ respectively $g \circ T^{t_{l_n}}$ converge
weakly to $f_1$ and $g_1$, respectively, by Lemma \ref{lem: weak
convergence compact}, $f \circ T^{t_{l_n}}$ respectively $g \circ
T^{t_{l_n}}$  are equi-bounded and converge pointwise to $f_1$ and
$g_1$, respectively. Then, it is immediate to see that $(f+g) \circ
T^{t_{l_n}}$ respectively $(f\cdot g) \circ T^{t_{l_n}}$
equi-bounded and converge pointwise to $f_1+g_1$ and $f_1 \cdot
g_1$, respectively. It now follows from Corollary \ref{cor: WAP by
sequences} that $f+g, f \cdot g \in WAP(\XX)$.

It is trivial to check that $WAP(\XX)$ is closed under scalar
multiplication, and it contains the constant function $1$.
\end{proof}

\begin{proposition}\label{prop: WAP is closed} $WAP(\XX)$ is closed subspace of $(C(\XX), \| \, \|_\infty)$.
\end{proposition}
\begin{proof}
Since $(C(\XX), \| \, \|_\infty)$ is a Banach space, it suffices to
work with sequences.

Let $(f_n)$ be a sequence in $ WAP(\XX)$ such that, in the sup norm,
$f_n \to f \in C(\XX)$. We need to show that $f \in \WAP(\XX)$. We
prove this by using Corollary \ref{cor: WAP by sequences}: Let
$(t_m)$ in $G$ be arbitrary. Then, we can pick some subsequence
$(t_m^{1})$ of $(t_m)$ such that $f_1 \circ T^{t_m^{1}}$ is weakly
convergent to some $g_1 \in C(\XX)$.  By induction, for each $n \geq
2$ we can pick some subsequence $(t_m^{n})$ of $(t_m^{n-1})$ such
that $f_n \circ T^{t_m^{n}}$ is weakly convergent to some $g_n \in
C(\XX)$.  Consider now the  diagonal $s_m := t_m^m$, $m\in\NN$.
Then, for all $n$ we have
\[
w-\lim_m f_n \circ T^{s_m} =g_n \,.
\]
We show next that $g_n$ is Cauchy in $(C(\XX), \| \, \|_\infty)$.
Indeed
\begin{eqnarray*}
\| g_n - g_p \|_\infty &=& \sup_{\psi \in C(\XX)' ; \| \psi \|=1} \left| \psi(g_n) - \psi(g_p) \right| \\
& =& \sup_{\psi \in C(\XX)' ; \| \psi \|=1} \lim_m \left| \psi \left( f_n \circ T^{s_m} - f_p \circ T^{s_m}\right) \right|  \\
&=& \sup_{\psi \in C(\XX)' ; \| \psi \|=1} \limsup_m \| \psi \|
\left| (f_n  - f_p ) \circ T^{s_m} \right|
=  \| f_n  - f_p  \|_\infty  \\
\end{eqnarray*}
Therefore, as $(C(\XX), \| \, \|_\infty)$ is complete, the sequence
$(g_n)$ converges to some $g \in C(\XX)$.

We claim now that $f \circ T^{s_m}$ converges weakly to $g$. This
claim, once proven, completes the proof by Corollary \ref{cor: WAP
by sequences}.

Let $\epsilon >0$ and $\psi \in C(\XX)'$.  Since $g_n \to g$ and
$f_n \to f$, there exists some $N$ such that for all $n >N$ we have
\[
\| g_n -g \|_\infty < \frac{\epsilon}{\| \phi \| +1} \, \mbox{ and }
\| f_n -f \|_\infty < \frac{\epsilon}{\| \phi \| +1}
\]
Pick some $n_0 >N$. As $w-\lim_m f_{n_0} \circ T^{s_m} =g_{n_0}$
there exists some $M$ such that for all $m >M$ we have
\[
\left| \psi (f_{n_0} \circ T^{s_m} -g_{n_0} ) \right| < \epsilon \,.
\]
Then, for all $m > M$ we have
\begin{eqnarray*}
\left| \psi (g -f \circ T^{s_m}  ) \right|  &\leq&  \left| \psi (g -g_{n_0} ) \right|  + \left| \psi (g_{n_0} -f_{n_0} \circ T^{s_m}  ) \right| + \left| \psi (f_{n_0} \circ T^{s_m}  -f \circ T^{s_m}  ) \right| \\
  &\leq& \| \psi\| \| g-g_{n_0}\|_\infty + \epsilon + \| \psi\| \| f-f_{n_0}\|_\infty   < 3 \epsilon \,.
\end{eqnarray*}
This proves the claim.
\end{proof}
We complete our considerations on general dynamical systems by
proving the following result. Part (a) can also be found in
{\cite[Thm.~1.43 (4)]{GLAS}

\begin{theorem}\label{thm: wap iff wap}
Let $(\XX, G)$ be a dynamical system. Let  $f \in C(\XX)$ and
$\omega \in \XX$ be given and define $g : G \to \CC$ via $ g(t)= f
(T^t \omega)$. Then, the following holds:
\begin{itemize}
  \item [(a)] If $f \in WAP(\XX)$ then $g \in WAP(G)$.
  \item [(b)] If $g \in WAP(G)$ and $\{ T^t \omega \mid  t\in G\}$ is dense in $\XX$ then $f \in WAP(\XX)$.
\end{itemize}
\end{theorem}
\begin{proof}
Before proving (a) and (b) we need a little preparation. Define
\[
F: C(\XX) \to \Cu(G)  \,;\, F(h)(s)= h(T^s(\omega)) \,.
\]
It is easy to see that $F$ is well defined. Indeed, as $\XX$ is
compact, $h$ is bounded and hence so is $F(h)$. Moreover, $h$ is
uniformly continuous, and by the continuity of the group action on
$\XX$ so is $F(h)$. Moreover,
\begin{equation}\label{eq: norm ineq}
\| F( h) \|_\infty  = \sup_{s \in G} \left|  h(T^s(\omega))  \right|
\leq  \| h \|_\infty \,.
\end{equation}
Let us also observe that if $\{ T^t \omega  t\in G\}$ is dense in
$\XX$ then we have equality in (\ref{eq: norm ineq}).

Now we proceed to prove (a) and (b).

(a) $F$ is a continuous operator between the Banach spaces $
(C(\XX), \| \, \|_\infty)$ and $(\Cu(G), \| \, \|_\infty)$ and hence
it is clearly  weakly continuous (see also  \cite{MoSt} for a recent
discussion of this type of reasoning). Since $\WW(f)$ is weakly
compact, and $F$ is weakly continuous, it follows that $F(\WW(f))$
is weakly compact in $\Cu(G)$. Now, by construction we have $T^t
F(f) = F(T^{-t} f)$. Therefore
\[
\{ T^t F(f) | t \in G \} \subset F(\WW(f)) \,,
\]
and hence $g=F(f) \in WAP(G)$.

(b) Since $\{ T^t \omega  t\in G\}$ is dense in $\XX$, we have
equality in (\ref{eq: norm ineq}). Therefore, $F$ gives an isometric
embedding of $C(\XX) \hookrightarrow \Cu(G)$.  Denote by
\[
\YY:= \{ F(h)| h \in C(\XX) \} \subset \Cu(G) \,.
\]
Then, $F$ is an isometry between $C(\XX)$ and $\YY$. Since $C(\XX)$
is a Banach space, so is $\YY$ and therefore $\YY$ is closed in
$\Cu(G)$. As $F$ is linear, $\YY$ is a subspace of $\Cu(G)$. As a
closed subspace $\YY$ is also weakly closed in $\Cu(G)$.

Consider now the isometry $F^{-1} : \YY \to C(\XX)$. Since
\[
T^t F(f) = F(T^{-t} f) \,,
\]
we have $T^t g \in \YY$. As $\YY$ is weakly closed, and $g \in
WAP(G)$ it follows that weak closure $\overline{ \{ T^t g |t \in G
\} }^{w}$ of the orbit of $g$ is compact and  contained $\YY$. Then,
as $F^{-1}$ is continuous, it is weakly continuous, and therefore
$F^{-1}(J)$ is weakly compact. As this set contains the orbit of
$f$, we get that $f \in WAP(\XX)$.
\end{proof}

We can now come to the main result of this section.

\begin{theorem}\label{thm:wap hull if and only if wap measure}
Let $\mu \in \cM^\infty (G)$ be given. Then, $\mu$ is weakly almost
periodic if and only if $(\XX (\mu), T)$ is a weakly almost periodic
dynamical system.
\end{theorem}
\begin{proof} This follows from  Theorem \ref{thm: wap iff wap}
applied with the functions  $\phi_c : \cM^\infty (G) \longrightarrow
\CC, \phi_c (\mu) = c * \mu (0)$ for $c\in \Cc (G)$. Note that
$\phi_c(T^t \mu) = c*\mu(t)$. Here are the details:

Assume that $(\XX(\mu),G)$ is weakly almost periodic. Let $c \in
\Cc(G)$.  Since $\phi_c \in WAP(\XX(\mu))$ it follows from Theorem
\ref{thm: wap iff wap} (a) that $c *\mu \in WAP(G)$.  As this is
true for all $c \in \Cc(G)$ we get $\mu \in \WAP(G)$.

Assume conversely that $\mu$ is a weakly almost periodic function.
If $c \in \Cc(G)$ then it follows from Theorem \ref{thm: wap iff
wap} (b) that $\phi_c \in WAP(\XX(\mu))$. This shows that $\phi_c
\in WAP(\XX(\mu))$ for all $c \in \Cc(G)$.  Then, by Proposition
\ref{prop: WAP is algebra}, $WAP(\XX(\mu))$ contains the algebra
generated by $\{ \phi_c | c \in \Cc(G) \}$. Since this algebra is
dense in $C(\XX(\mu))$ \cite{BL} and since $WAP(\XX(\mu))$ is a
closed subspace of $(C(\XX(\mu)), \| \, \|_\infty)$, it follows that
\[
C(\XX(\mu))=WAP(\XX(\mu)) \,.
\]
This finishes the proof.
\end{proof}

We now derive a few corollaries from this concerning strongly almost
periodic systems and measures.  Here, a dynamical system $(\XX,T)$
is called \textbf{strongly  almost periodic} if for any $f\in C
(\XX)$ the set $\{ f\circ T^t : t\in G\}$ is relatively compact in
$(C(\XX),\|\cdot\|_\infty)$. We will need the following  well-known
statements on strongly almost periodic systems.

\begin{proposition}\label{prop:wap-vs-sap}

(a) Any minimal component of a weakly almost periodic dynamical
system is strongly almost periodic.

(b) For a  transitive dynamical system $(\XX,G)$ with dense orbit of
$\omega\in\XX$ the following statements are equivalent:

\begin{itemize}

\item[(i)] $(\XX,G)$ is strongly almost periodic.

\item[(ii)] $(\XX,G)$  admits a structure of a locally compact group such
that $G\longrightarrow \XX, t\mapsto T^t \omega$ becomes a
continuous group homomorphism.

\item[(iii)] $(\XX,G)$ is weakly almost periodic and minimal.
\end{itemize}

\end{proposition}

\begin{proof}
(a) By \cite[Prop.~2.8]{EN} any minimal weakly almost periodic
dynamical system is strongly almost periodic.  Now, (a) follows as
any minimal component of a weakly-almost periodic system is also
weakly almost periodic. (Indeed, if  $\YY$ is an invariant compact
subset  of $\XX$ then any $f\in C(\YY)$ can be extended by
compactness to some $g \in C(\XX)$. Define now $R :C(\XX) \to
C(\YY)$ to be the restriction of the domain to $\YY$. Then we have
$\|R \| \leq 1$, which shows that $R$ is continuous, hence weakly
continuous. Therefore, as $g \in C(\XX)=WAP(\XX)$ we have $f=R(g)
\in WAP(\YY)$.)

(b) The equivalence of (i) and (ii)   can be found e.g. in the book
of Auslander \cite{Aus}. The implication
(iii)$\Longrightarrow$ (i) is immediate from (a). On the other hand
(i) clearly implies weak almost periodicity and (ii) clearly implies
minimality of $(\XX,G)$ and so  (iii) follows from (i) and (ii).
\end{proof}

Here are the corollaries.

\begin{corollary} \label{saphull}
A translation bounded measure $\mu$ belongs to $\SAP (G)$ if and
only if $(\XX(\mu),G)$ is strongly almost periodic.
\end{corollary}
\begin{proof}   As is well-known, see e.g. \cite{LR} for a recent
discussion, $(\XX (\mu),G)$ admits a group structure such that
$t\mapsto T^t \mu$ is a continuous group homomorphism  if and only
if $\mu$ is strongly almost periodic. Now, the statement follows
from the equivalence of (i) and (ii) in (b) of Proposition
\ref{prop:wap-vs-sap}.
\end{proof}

\begin{corollary}\label{wap+minimal implies}
Let $\mu \in \WAP(G)$. Then,
 $\XX(\mu)$ is minimal if and only if $\mu \in \SAP(G)$.
\end{corollary}
\begin{proof} By Theorem \ref{thm:wap hull if and only if wap
measure} $(\XX(\mu),G)$ is weakly almost periodic.  Hence, by the
equivalence between (i) and (iii) in (b) of  Proposition
\ref{prop:wap-vs-sap}, it is strongly almost periodic if and only if
it is minimal. Now, the desired statement follows from the preceding
corollary.
\end{proof}

\subsection{Stability of almost periodicity along the
hull}\label{sec:stability} In this section we show that all elements
in the hull of a strongly, weakly and null-weakly almost periodic
measure are strongly, weakly and null-weakly almost periodic
respectively.

\smallskip

We begin with the following well-known consequence of Lemma
\ref{compact}.
\begin{lemma}\label{lem: strong compactness implies the same hull}
Let  $\tau_1$ and $\tau_2$ be topologies on $X$ such that $\tau_1$
is stronger than $\tau_2$ and $\tau_2$ is Hausdorff. If $A \subset
X$ is precompact with respect to $\tau_1$, then the closures of $A$
with respect to $\tau_1$ and $\tau_2$ coincide, and the two
topologies coincide on this closure.
\end{lemma}

Here is the first stability result of this section.

\begin{proposition}\label{prop: hull of AP is AP}
Let  $\mu \in \cM^\infty (G)$ be given.

(a) The following assertions are equivalent:

\begin{itemize}
\item[(i)] $\mu\in \WAP (G)$.

\item[(ii)]  The hull $\XX (\mu)$ agrees with the closure of $\{T^t
\mu \mid t\in G\}$ in the weak topology and is, in particular,
compact in the weak topology.

\item[(iii)] The vague and the weak topologies agree on $\XX (\mu)$.

\item[(iv)] $\XX (\mu)\subset \WAP (G)$.

\end{itemize}

(b) The following assertions are equivalent:

\begin{itemize}
\item[(i)] $\mu\in \SAP (G)$.

\item[(ii)]  The hull $\XX (\mu)$ agrees with the closure of $\{T^t
\mu \mid t\in G\}$ in the product topology and is, in particular,
compact in the product  topology.

\item[(iii)] The vague and the product  topology  agree on $\XX (\mu)$.

\item[(iv)] $\XX (\mu)\subset \SAP (G)$.

\end{itemize}

\end{proposition}
\begin{proof} (a) (i)$\Longrightarrow$ (ii): As $\mu$ is weakly
almost periodic, the closure of its orbit in the weak topology is
compact. Now, (ii) follows from the previous Lemma  as the weak
topology is stronger than the vague topology.

(ii)$\Longrightarrow$ (iii): This follows from Lemma \ref{compact}.

(iii)$\Longrightarrow$ (iv): Let  $\nu \in \XX(\mu)$ be arbitrary.
We then  have $\XX(\nu) \subset \XX(\mu)$. As the vague and the weak
topology agree on $\XX (\mu)$ we infer that $\XX (\nu)$ is compact
in the weak topology. Hence, $\nu$ belongs to $\WAP (G)$. This gives
(iv).

(iv)$\Longrightarrow$ (i): This is clear.

\smallskip

(b) This can be shown by analogous arguments.
\end{proof}

We next show that  for weakly almost periodic measures, the Fourier
Bohr coefficients depend continuously on the measure.

\begin{theorem}\label{thm:bragg peaks have continuous eigenfunctions}
Let $\mu \in \WAP(G)$ and $\chi \in \widehat{G}$. Then $c_\chi :
\XX(\mu) \to \CC $ is continuous and satisfies
\[
c_\chi (T^t \nu) = \chi(t) c_\chi(\nu) \,.
\]
In particular, if $c_\chi(\mu) \neq 0$, then $c_\chi$ is a
continuous eigenfunction of $\XX(\mu)$.
\end{theorem}
\begin{proof}
Since $\mu \in \WAP(G)$ the hull $\XX(\mu)$ is weak compact. The
continuity of $c_\chi$ follows now from Theorem \ref{thm: cont FB
coeff}.

Moreover, for all $t \in G$ we have
$$ c_\chi (T^t \nu)   = M( \bar{\chi}T^t \nu) = M(T^{t} \left( (T^{-t}\bar{\chi}) \nu\right)) \\
 = M \left( \chi(t) (\bar{\chi}) \nu\right) = \chi(t) c_\chi(\nu)$$
and the proof is finished.
\end{proof}

\begin{remark}
As we will see later, the intensity of the Bragg peak at $\chi \in
\widehat{G}$ in the diffraction of $\mu$ is given by $\left|
c_\chi(\mu) \right|^2$. So Theorem \ref{thm:bragg peaks have
continuous eigenfunctions} gives that  each Bragg peak $\chi$, comes
from  a continuous eigenfunction on $\XX(\mu)$.
\end{remark}

\begin{proposition}\label{prop: hull of 0-WAP is 0-WAP}
If $\mu \in \WAP_0(G)$ then $\XX(\mu) \subset \WAP_0(G)$.
\end{proposition}
\begin{proof} This follows easily from the previous theorem and the
fact that a weakly almost periodic measure is null weakly almost
periodic if and only if all of its Fourier coefficients are zero
\cite[Thm.~8.1]{ARMA}. Specifically, as $\mu$ is null weakly almost
periodic all of its Fourier coefficients are zero. By the
formula given in the  previous theorem this then holds as well for
all elements in the orbit of $\mu$ and by the continuity statement
of the previous theorem, this then holds for the whole hull.
\end{proof}

\subsection{The unique minimal component and unique ergodicity} \label{sec:WAP hull}
In this section we show that the hull of a weakly almost periodic
measure has a unique minimal component and that this component is
just the hull of the strongly almost periodic part of the measure.
This has strong consequences. For example a weakly almost periodic
dynamical system is uniquely ergodic if and only if it has a unique
minimal component.

\smallskip

A first step (and main result in this section) is  that the strongly
almost periodic part $\mu_{\mathsf{s}}$ of $\mu$ belongs to the
hull. As a consequence we can then derive a wealth of results on
weakly almost periodic measures in this section and the next one.

As a preparation of the proof of the main result of this section we
show that for a uniformly continuous function $f$, whose average
absolute value integral is $0$,  we can find arbitrarily large
regions in $G$ where $|f|$ is arbitrarily small.

\begin{lemma}\label{wap_0 functions are arbitrarily small} \rm Let
$(A_n)$ be a van Hove sequence.  Let $f$ be an uniformly continuous
bounded function on $G$ such that
\[ \lim_n \frac{\int_{A_n} \left| f(x+t) \right| dx }{\theta_G(A_n)} = 0 \,.
\]
uniformly in $t \in G$. Then for all compact sets $K \subset G$ and
all $\epsilon >0$, there exists some $t \in G$ so that
\[
\left| f(x) \right| < \epsilon
\]
for all $x\in t + K$.
\end{lemma}

\begin{remark} Note that the lemma applies to
any null weakly almost periodic function.
\end{remark}

\begin{proof}
Assume the contrary. Then, there exists a compact $K$ and an
$\varepsilon>0$ such that  for any $t\in G$ there exists an $x_t\in
K$ with $|f( t+ x_t)| \geq \varepsilon$. As $f$ is uniformly
continuous, there exists then an open relatively compact set $U$
with
\[
|f(t + y)| \geq \frac{\varepsilon}{2}
\] for all $t\in G$ and
all $y\in x_t + U$. This implies, in particular,
\[
\int_{K + U} |f( t + x)| dx \geq \frac{\varepsilon}{2}
\theta_G (U)=:c >0
\]
for all $t\in G$ and,  hence,
\begin{equation}\label{eq-cont}
\liminf_n \frac{\int_{A_n} \int_{K+U} |f(t + x)| dx dt}{\theta_G
(A_n)} \geq c >0 \,.
\end{equation}
On the other hand, from the assumption on $f$
and as $(A_n)$ is van Hove we have
\[
\frac{\int_{A_n} |f(t + x)|
dx}{\theta_G (A_n)} \to 0,n\to \infty
\] uniformly in $t\in G$.

 From Fubini-theorem, we then obtain
\[
\frac{\int_{A_n} \int_{K+U} |f(t + x)| dx dt}{\theta_G(A_n)} = \int_{K+ U} \left(\frac{\int_{A_n} |f(t + x)| dx}{\theta_G
(A_n)} \right)dt
\]

Now, by the uniform convergence of $\frac{\int_{A_n} |f(t + x)| dx}{\theta_G
(A_n)}$, there exists some $N$ such that, for all $n >N$ we have
\[
\frac{\int_{A_n} |f(t + x)| dx}{\theta_G(A_n)} < \frac{c}{2 \theta_G(K+U)} \,.
\]
Then, for all $n >N$ we have
\[
\frac{\int_{A_n} \int_{K+U} |f(t + x)| dx dt}{\theta_G(A_n)} = \int_{K+ U} \left(\frac{\int_{A_n} |f(t + x)| dx}{\theta_G
(A_n)} \right)dt  \leq  \int_{K+ U} \frac{c}{2 \theta_G(K+U)} dt \leq \frac{c}{2} \,.
\]
which contradicts (\ref{eq-cont}). This completes the proof.
\end{proof}

\begin{remark} Let us shortly discuss the assumptions in the lemma:
The uniform continuity of $f$ is essential in Lemma \ref{wap_0
functions are arbitrarily small}, as can be seen by considering the
following example:  For each $n \in \ZZ$ pick some $f_n \in
\Cc(\RR)$ such that $f_n \geq 0, f_n(n)=1, \supp(f_n) \subset
(n-\frac{1}{2}, n+\frac{1}{2})$ and $\int_\RR f_n =
\frac{1}{2^{|n|}} \,.$ Let
\[
f:= \sum_{n \in \ZZ} f_n \,.
\]
Then $f$ is non-negative and  has finite integral. Hence, it has
zero average integral. However, clearly any translate of $[-1,1]$
will meet an element where $f$ takes the value $1$.

The assumptions of boundedness of $f$ and of uniform existence of
the limit can be removed if vanishing of the limit is known for any
van Hove sequence.
\end{remark}

We are now ready to prove  the main result of this section viz that
the hull of a weakly almost periodic function contains its strong
almost periodic part.

\begin{theorem}\label{thm:Hull WAP}
Let $\mu$ be a weakly almost periodic measure. Then
\[
\mu_{\mathsf{s}} \in \XX( \mu) \,.
\]
In particular, $\XX (\mu_\mathsf{s})\subset \XX (\mu)$.
\end{theorem}

\begin{proof}
We need to show that for every natural number $n$ and all
$c_1,\ldots, c_n \in \Cc(G)$ and $\epsilon
>0$ there exists some $r \in G$ so that for all $1 \leq i \leq n$ we
have
\[\left| (T^r\mu)(c_i)- \mu_{\mathsf{s}}(c_i) \right| < \epsilon \,.\]
The basic idea of the proof is now to chose a compact $K$ so that
any translate of $K$ contains a vector $r$ such that $T^r
\mu_{\mathsf{s}}$ is very close to $\mu_{\mathsf{s}}$ (which is
possible by almost periodicity of $\mu_{\mathsf{s}}$) and then to
shift this $K$ so that $T^r \mu_{\mathsf{0}} (c_i)$ is very small
for all $r$ in the shifted $K$ (which is possible by the previous
lemma).

Here, are the details:  Let $c_1,...,c_n \in \Cc(G)$ and $\epsilon
>0$ be given.  Since the functions $\mu_{\mathsf{s}}*c_i^\dagger$
are almost periodic, the set
\[P:= \{ t \in G | \max_{1 \leq i \leq n } \|\mu_{\mathsf{s}}*c_i^\dagger - T^t \mu_{\mathsf{s}}*c_i^\dagger \|_\infty < \frac{\epsilon}{2}\} \,,\]
is relatively dense. Fix a compact $K$ such that $0 \in K$ and
$P+K=G$. Define
\[
g_0= \sum_{i =1}^n  \left| \mu_0*c_i^\dagger \right| \,.\] Then
$g_0$ is a null weakly almost periodic function. By Lemma \ref{wap_0
functions are arbitrarily small}, there exists then  some $t \in G$
so that
\[\left| g_0(x) \right| < \frac{\epsilon}{2}\]
for all $x \in t-K$.  Since $t \in G = P+K$, there exists some $r
\in P$ and $k \in K$ such that $t= r+k$.  Therefore
\[
\left| g_0 (r) \right|= \left| g_0(t-k) \right| < \frac{\epsilon}{2} \,.\]
and hence, by the definition of $g_0$, for all $1 \leq i \leq n$ we have
\[\left| \mu_0*c_i^\dagger \right|(r) < \frac{\epsilon}{2} \,.\]
Moreover, as $r \in P$ we also have \[ \| \mu_{\mathsf{s}}*c_i^\dagger - T^r
\mu_{\mathsf{s}}*c_i^\dagger \|_\infty < \frac{\epsilon}{2} \,.\]

Combining these two relations we get:
\[
\left| \mu_{\mathsf{s}}*c_i^\dagger(0)-
[T^r\mu_{\mathsf{s}}+T^r\mu_0]*c_i^\dagger(0) \right| < \epsilon
\,.\] Therefore
\[ \left| \mu_{\mathsf{s}}( c_i)- [T^r\mu]( c_i) \right| < \epsilon \,.\]
which completes the proof.
\end{proof}

\begin{corollary} Let $\mu$ be a weakly almost periodic measure.
Then, $\XX (\mu_\mathsf{s})$ is the unique minimal component of
$\XX(\mu)$.
\end{corollary}
\begin{proof} By definition  $\mu_{\mathsf{s}}$ is strongly almost
periodic. Hence, its hull $\XX (\mu_{\mathsf{s}})$ is minimal.
Moreover, by the previous theorem, this hull is part of $\XX(\mu)$.
It remains to show the uniqueness part of the statement: Let  $\nu$
be any element of $\XX (\mu)$  belonging  to a minimal component.
Then,  it is strongly almost periodic by abstract principles (see
above). Hence, $\nu = \nu_s$.  Now, by assumption we have
$T^{t_\beta} \mu \to \nu$ for a net $(t_\beta)$. Without loss of
generality we can assume that $T^{t_\beta} \mu_0$ converges to, say,
 the measure $\alpha $ and $T^{t_\beta} \mu_s$ converges to, say, the measure  $\rho$ say.
Then, we have
$$ \alpha + \rho =  \nu = \nu_s.$$
By Proposition \ref{prop: hull of 0-WAP is 0-WAP} we have
that $\alpha \in \WAP_0(G)$ and clearly $\rho \in \SAP(G)$. Thus, by uniqueness of the decomposition,
we have $\alpha = 0$. Hence, we have  that $\nu = \nu_s = \rho$
belongs to the hull of $\mu_s$. This finishes the proof.
\end{proof}

An immediate consequence is the following simple characterisation of
null almost periodicity. To state it we will use the notation
$\underline{0}$ to denote the zero measure on $G$.

\begin{corollary}
Let $\mu \in \WAP(G)$. Then $\mu$ is null weakly almost periodic if
and only if the zero measure $\underline{0}$ belongs to $ \XX(\mu)$.
In this case $\delta_{\underline{0}}$ is the uniquely ergodic
measure on $\XX(\mu)$.
\end{corollary}

As is well-known, for weakly almost periodic systems unique
ergodicity is equivalent to uniqueness of the minimal component.
Indeed, each minimal component certainly carries  ergodic measure
and conversely uniqueness of such a component implies unique
ergodicity  \cite[Prop.~2.10]{EN}. So from the previous corollary
and Theorem \ref{thm:wap hull if and only if wap measure} we
directly infer the following result.

\begin{corollary}\label{cor:wap are ue} Let $\mu$ be a weakly almost periodic measure.
Then, $(\XX (\mu),G)$ is uniquely ergodic.
\end{corollary}

\begin{remark} We note that unique ergodicity can also  be
derived  from existence of means of weakly almost periodic
functions: For any $c_1,..,c_n\in C_c (G)$ and  the function $t\to
\prod_{j=1}^n \phi_{c_j} (T^t \mu)$ is weakly almost periodic.
Hence, its mean exists in a rather uniform manner. Since such linear
combinations form a dense subset of $C(\XX(\mu))$. one obtains
uniform existence of means for all functions in $C (\XX (\mu))$.
This, in turn, implies unique ergodicity, see \cite{LS4}.
\end{remark}

\subsection{The structure of the hull}\label{sec:Hull structure} In this section we look a  bit closer at
the relation between $\XX(\mu)$ and $\XX(\mu_{\mathsf{s}})$. This
will allow us to get a rather precise description of the hull $\XX
(\mu)$.

\bigskip

The dynamical system $(\XX(\mu_{\mathsf{s}}),G)$ will play a central
role in subsequent considerations. Hence, we will give it a special
name.

\begin{definition} Let $\mu \in \WAP$. We define
$\SSS(\mu):= \XX(\mu_{\mathsf{s}})$.
\end{definition}

As $\mu_{\mathsf{s}}$ is strongly  almost periodic, we infer from
Corollary  \ref{saphull} that $\SSS(\mu)$ is a compact abelian
group. We denote the Haar measure on this group by
$\theta_{\SSS(\mu)}$. Moreover, by Theorem \ref{thm:Hull WAP} we
have that $\SSS(\mu)$ is a subset of $\XX(\mu)$. Clearly, it is a
closed $G$-invariant subset. Hence, from the unique ergodicity of
$(\XX(\mu), G)$ and the fact that $\theta_{\SSS(\mu)}$ is an
invariant measure, we obtain the following corollary.

\begin{corollary}\label{more on dynamical hulls} Let $\mu \in \WAP(G)$.
The Haar measure $\theta_{\SSS(\mu)}$ is the uniquely ergodic
measure of $\XX(\mu)$.
\end{corollary}
\begin{remark} Of course, this is  just a special case of the
general fact that a weakly almost periodic system with a unique
minimal component has the Haar measure on this component as its
unique translation invariant measure.
\end{remark}

We next  show that $\SSS(\mu)$ consists exactly of the strong almost
periodic measures in $\XX(\mu)$. First we present a more general
result.

\begin{lemma}
Let $\XX$ be a weakly almost periodic dynamical system which has an
unique minimal component $\SSS$. Then, for an element $\omega \in
\XX$ we have $\omega \in \SSS$ if and only if $\XX(\omega)$ is an
almost periodic dynamical system.
\end{lemma}
\begin{proof} Assume $\omega\in \SSS$.
Since $\SSS$ is  a minimal component, we have $\XX(\omega)=\SSS$. By
the minimality of $\SSS$ we get that $\SSS$ is an almost periodic
dynamical system.

Assume now that  $\XX(\omega)$ is an almost periodic dynamical
system. Hence, it must be  minimal. Hence, by the uniqueness of the
minimal component we get $\XX(\omega)=\SSS$ and hence $\omega\in
\SSS$.
\end{proof}

As a corollary we obtain the following.

\begin{lemma}\label{hull intersected sap} Let $\mu \in \WAP(G)$.
Then $\XX(\mu) \cap \SAP(G) = \SSS(\mu)$.
\end{lemma}

Next we show that the mapping $P_0 : \WAP(G) \to \WAP_0(G)$ takes
$\XX(\mu)$ into $\XX(\mu_0)$.

\begin{lemma}\label{nu0 in hull mu0} Let $\mu \in \WAP(G)$. If $\nu \in \XX(\mu)$ then
$ \nu_0 \in \XX(\mu_0)$.
\end{lemma}
\begin{proof}
Let $\nu \in \XX(\mu)$. Then there exists some net $(t_\alpha)$ in $
G$ such that
\[
\nu = \lim_\alpha T_{t_\alpha} \mu \,.
\]
Next, let us look at $T_{t_\alpha} \mu_{\mathsf{s}}$. As
$T_{t_\alpha} \mu_{\mathsf{s}} \in \XX(\mu_{\mathsf{s}})$ and
$\XX(\mu_{\mathsf{s}})$ is compact, the net has cluster points.
Thus, we can find some subnet $(\beta)$ of $(\alpha)$ and some
$\omega \in \XX(\mu_{\mathsf{s}})$ such that
\[
T_{t_\beta} \mu_{\mathsf{s}} \to \omega  \,.
\]

As $T_{t_{\beta}} \mu_{\mathsf{s}} \in \XX(\mu_{\mathsf{s}})$ we get
$\omega \in \XX(\mu_{\mathsf{s}})$. Therefore, by Proposition
\ref{prop: hull of AP is AP} we get that $\omega \in \SAP(G)$.

Now, as $T_{t_{\beta}} \mu_{0} \in \XX(\mu_{0})$ and $\XX(\mu_0)$ is
compact, we can find a subnet $\gamma$ of $\beta$ such that
\[
T_{t_{\gamma}} \mu_{0} \to \upsilon \in \XX(\mu_0) \,.
\]

Then, by Proposition \ref{prop: hull of 0-WAP is 0-WAP} we have $\upsilon \in
\WAP_0(G)$.

Now, we have
\begin{eqnarray*}
  \nu  &=&  \lim_n T_{\alpha} \mu =\lim_n T_{t_{\gamma}} \mu =\lim_n T_{t_{\gamma}} ( \mu_{\mathsf{s}}+ \mu_0)   \\
  &=& \lim_n T_{t_{\gamma}} \mu_{\mathsf{s}}+ \lim_n T_{t_{\gamma}}\mu_0  = \omega+ \upsilon \,.
\end{eqnarray*}
The uniqueness of the decomposition gives us that
$\upsilon = \nu_0$. Since $\upsilon \in \XX(\mu_0)$ we get the
claim.
\end{proof}

\smallskip

Now, we can prove that the projections $P_{\mathsf{s}}$ and $P_0$
are continuous when restricted to the hull $\XX(\mu)$.

\begin{lemma}\label{continuity of projections on hull}
Let $\mu \in \WAP(G)$.
Then the maps
\[
P_{\mathsf{s}}: \XX(\mu) \to \SSS(\mu), \nu \mapsto \nu_{\mathsf{s}}
\mbox{ and } P_0 : \XX(\mu) \to \XX(\mu_0), \nu \mapsto\nu_0,
\]
are continuous.
\end{lemma}
\begin{proof}

First, let us observe that $P_{\mathsf{s}}$ is well defined i.e.
maps into $\SSS (\mu)$. Indeed, if $\nu \in \XX(\mu)$ we know that
$\nu \in \WAP(G)$ and hence
\[
P_s(\nu) \in \SAP(G) \cap \XX(\nu) \subset \SAP(G) \cap \XX(\nu)   =
\SSS(\mu) \,.
\]
 In the same way
$P_0$ is well defined.

We next prove the continuity of $P_{\mathsf{s}}$. To do this we have
to show that whenever $\nu_\alpha \to \nu$ in $\XX(\mu)$ we have
\[
(\nu_{\alpha})_{\mathsf{s}} \to \nu_{\mathsf{s}} \,.
\]

As $\XX(\mu_{\mathsf{s}})=\SSS(\mu)$ is compact, and
$(\nu_\alpha)_{\mathsf{s}},\nu_{\mathsf{s}} \in
\XX(\mu_{\mathsf{s}})$, the net $(\nu_\alpha)_{\mathsf{s}}$ as well
as any subnet has cluster points. To prove that
$(\nu_{\alpha})_{\mathsf{s}} \to \nu_{\mathsf{s}}$, it thus suffices
to show that every cluster point of $(\nu_{\alpha})_{\mathsf{s}}$ is
equal to $\nu_{\mathsf{s}}$.

Let $\omega$ be a cluster point of $(\nu_\alpha)_{\mathsf{s}}$.
Then, there exists a subnet $\beta$ of $\alpha$ such that
\[
\omega= \lim_{\beta} (\nu_\beta)_{\mathsf{s}} \,.
\]

As $(\nu_\beta)_{\mathsf{s}} \in \SAP(G) \cap \XX(\mu) = \SSS(\mu)$
which is compact, hence complete in the vague topology, we get
$\omega \in \SSS(\mu)$ and hence, $\omega \in \SAP(G)$.

Moreover, by Lema \ref{nu0 in hull mu0}, as $\nu_\beta \in \XX(\mu)$
we have $(\nu_\beta)_0 \in \XX(\mu_0)$.

As before, by compactness of $\XX(\mu_0)$ we can also pick a subnet
$(\nu_\gamma)_0$ such that $(\nu_{\gamma})_{0}$ is convergent in
$\XX(\mu_0)$. Let $\upsilon$ be the limit of this net, then
$\upsilon \in \WAP_0(G)$.

Then
\[
\nu=\lim_\alpha \nu_\alpha = \lim_\gamma \nu_{\gamma}=  \lim_\gamma
(\nu_{\gamma})_{\mathsf{s}}+ (\nu_{\gamma})_0 = \omega+ \upsilon \,.
\]

As $\omega \in \SAP(G)$ and $\upsilon \in \WAP_0(G)$, the uniqueness
of the decomposition shows that
\[
\nu_{\mathsf{s}}=\omega \,.
\]
This shows that $P_s$ is continuous.

Now, consider $P_0: \XX(\mu) \to \WAP(G)$. Then, we have
\[
P_0(\nu) = \nu -P_{\mathsf{s}}(\nu) \,,
\]
is the difference of two functions which are continuous on
$\XX(\mu)$, and hence $P_0$ is continuous on $\XX(\mu)$. Restricting
now the codomain to $\XX(\mu_0)$ will not change the continuity, and
prove the claim.
\end{proof}

\smallskip

\begin{remark} For the preceding results the restriction to the hull
of one element is crucial.  On the whole space $\WAP (G)$ the maps
$P_{\mathsf{s}}$ and $P_0$ are not continuous with respect to the
vague topology. Indeed, for example for $G = \RR^N$ it suffices to
consider the sequence $\mu_n = \delta_ {n\ZZ^N}$. Then, each $\mu_n$
is of fully periodic and hence belongs to $ \SAP(G)$. However, the
vague limit of the $\mu_n$ is given by $\mu = \delta_0$.  So,  in
this case we then have
\[
P_{\mathsf{s}}(\mu_n) \to \mu \neq 0= P_{\mathsf{s}}(\mu) \, \mbox{
and } \, P_{0}(\mu_n)=0 \to 0 \neq \mu= P_{0}(\mu) \,.
\]
\end{remark}

We summarize part of the preceding considerations as follows.

\begin{theorem}\label{Hull decomp} Let $\mu \in \WAP(G)$. Then
\begin{itemize}
  \item [(a)] $\SSS(\mu)$ is a compact abelian group with Haar measure $\theta_{\SSS(\mu)}$.
  \item [(b)] $\underline{0} \in \XX(\mu_0)$ and $\delta_{\underline{0}}$ is the uniquely ergodic measure on $\XX(\mu_0)$.
  \item [(c)] The mapping
\begin{displaymath}
S: \XX(\mu) \to \SSS(\mu) \times \XX(\mu_0) \quad ;\quad S(\mu)=
(\mu_{\mathsf{s}}, \mu_0)
\end{displaymath}
is a continuous $G$-mapping, one to one, and has full measure range.
\end{itemize}
\begin{proof} The only thing which was not proven is $(c)$.
 The continuity of $S$ is an immediate consequence of Theorem
\ref{continuity of projections on hull}.  The $G$-invariance of $S$
follows immediately from the uniqueness of the Eberlein
decomposition:
\[
T_t \mu = T_t(\mu_{\mathsf{s}} + \mu_0)= T_t (\mu_{\mathsf{s}}) +
T_t( \mu_{0})
\]
is the Eberlein decomposition of $T_t \mu$ and therefore
\[
T_t (\mu_{\mathsf{s}})= (T_t \mu)_{\mathsf{s}} \quad ; \quad T_t
(\mu_{0})= (T_t \mu)_{0} \,.
\]
The fact that the $S$ is one to one is obvious by the Eberlein
decomposition.  Finally, we have that
\[
\SSS(\mu) \times \{ 0\} = S( \SSS(\mu)) \subset S(\XX(\mu)) \,,
\]
has full measure in $\SSS(\mu) \times \XX(\mu_0)$.
\end{proof}
\end{theorem}

From these considerations we also easily infer the following result.

\begin{theorem}\label{thm:structure}
Let $\mu \in \WAP(G)$. Then the mappings $i : \SSS(\mu)
\hookrightarrow \XX(\mu), i(\nu)=\nu$ and $P_{\mathsf{s}} :\XX(\mu)
\to \SSS(\mu)$ are well defined, continuous $G$-mappings, and
\begin{eqnarray*}
P_{\mathsf{s}} \circ i (\nu) &=& \nu \mbox{ for all } \nu \in \SSS(\mu) \, \mbox{and} \\
i \circ P_{\mathsf{s}} (\nu) &=& \nu \mbox{ for }
\theta_{\SSS(\mu)}- \mbox{ almost all } \nu \in \XX(\mu) \,.
\end{eqnarray*}
\end{theorem}
\begin{proof}

The mappings are well defined and continuous by Lemma \ref{continuity of projections on hull}.

Now, if $\nu \in \SSS(\mu)$ then $\nu \in \SAP(G)$. Therefore
\[
P_{\mathsf{s}} \circ i (\nu) = P_{\mathsf{s}} (\nu)= \nu_{\mathsf{s}}=\nu \,.
\]
Which proves the first relation.
Also, for all $\nu \in \XX(\mu)$ we have
\[
i \circ P_{\mathsf{s}} (\nu) = \nu_{\mathsf{s}} \,.
\]

Therefore $i \circ P_{\mathsf{s}} (\nu) = \nu$ for all $\nu \in \SSS(\mu)$. As $\theta_{\SSS(\mu)}(\SSS(\mu))=1$, we get that
\[
i \circ P_{\mathsf{s}} (\nu) = \nu \mbox{ for }
\theta_{\SSS(\mu)}- \mbox{ almost all } \nu \in \XX(\mu) \,.
\]
\end{proof}

\section{Spectral and diffraction theory of weakly almost periodic measures}\label{sec:dynam spectrum}
Based on the discussion in the  preceding section we can rather
directly set up the diffraction theory for a weakly almost periodic
measures. In fact,  the  preceding section has immediate consequence
for the spectral theory of the hull of the measure and this, in
turn, gives results on the diffraction.

\bigskip

We start with the following rather direct consequence of Theorem
\ref{thm:structure}.

\begin{proposition} Let $\mu \in \WAP(G)$. Then,
$$U : L^2 (\SSS (\mu),\theta_{\SSS(\mu)})\longrightarrow L^2 (\XX (\mu),\theta_\mu), g\mapsto g\circ P_{\mathsf{s}}$$
is a $G$-invariant unitary map mapping  $C(\SSS (\mu))$ into $C (\XX
(\mu))$.
\end{proposition}

From this proposition we can infer that the spectral theory of the
two dynamical systems $(\SSS (\mu),G)$ and $(\XX (\mu),G)$ is the
same. As the spectral theory of $(\SSS (\mu),G)$ is well-known we
then obtain a precise description of the spectral theory of $(\XX
(\mu),G)$. To make this explicit we need a bit more notation.

Let $(\XX,G)$ be a dynamical system and $\mu$ an $G$-invariant
probability measure on $G$. Then,  an $\chi \in \widehat{G}$  is
called an \textbf{eigenvalue} if there exists an  $f\in L^2
(\XX,\mu)$ with $f\neq 0$ and
$$ f\circ T^t = \chi (t) f$$
for all $t\in G$. Such an $f$ is then called an
\textbf{eigenfunction}. A dynamical system is said to have
\textbf{pure point dynamical spectrum} if there exists an
orthonormal basis of $L^2 (\XX,\mu)$ consisting of eigenfunctions.

\begin{theorem}\label{cor:cont eigenfunctions} Let $\mu \in \WAP(G)$.
Then the dynamical system $(\XX(\mu),G)$ has pure point dynamical
spectrum with an orthonormal basis of continuous eigenfunctions
given by
$$f_\lambda = \lambda \circ P_{\mathsf{s}}$$
for $\lambda \in \widehat{\SSS(\mu)}$.
\end{theorem}

\begin{remark}

(a) The group homomorphism $G \longrightarrow \SSS(\mu), t\mapsto
T^t \mu$, has dense range. Hence, its dual map $j:
\widehat{\SSS(\mu)} \longrightarrow \widehat{G}$ is injective and in
this way the elements of $\widehat{\SSS (\mu)}$ can be seen as
elements of $\widehat{G}$. A $\lambda\in\widehat{\SSS(\mu)}$ then
gives rise to the eigenvalue $j (\lambda)$.

(b) The result implies that $(\SSS (\mu),G)$ is what is called the
maximal equicontinuous factor of $(\XX(\mu),G)$. Indeed, one way to
define this factor is as the dual  to the group of all eigenvalues
with continuous eigenfunctions equipped with the discrete topology,
see e.g. \cite{ABKL} for a recent discussion in the context of
diffraction. Now,in our case this clearly  leads to $\SSS (\mu)$.

(c) Uniquely ergodic systems  with pure point spectrum with
continuous eigenfunctions are  called \textbf{isomorphic extensions
of their maximal  equicontinuous factor}. By the previous
corollary the hull of a weakly almost periodic measure is an example
of such a system. As shown recently by Downarowicz and Glasner
\cite{DG} isomorphy of the extension is equivalent to mean
equicontinuity in the minimal case.

(d) We have already met some  eigenfunctions in Theorem
\ref{thm:bragg peaks have continuous eigenfunctions}. Let us
emphasize that the eigenfunctions given in that theorem are
completely canonical. This is rather remarkable as usually
eigenfunctions are only determined up to some overall phase factor.
Note also that  Theorem
\ref{thm:bragg peaks have continuous eigenfunctions} does not necessarily provide an
eigenfunction for each eigenvalue (as the Fourier coefficients may
vanish for some group elements).
\end{remark}

\begin{proof}
 As
 $\SSS(\mu)$ is a compact group, it is  known to have pure
 point dynamical spectrum with an orthonormal basis of continuous
 eigenfunctions given by the elements $\lambda
 \in\widehat{\SSS(\mu)}$. Now, the statement follows from the
 previous proposition and the continuity of $P_s$.
\end{proof}

As a consequence we also obtain the following application to
diffraction theory. To express the corollary we will need the map
$$j: \widehat{\SSS(\mu)} \longrightarrow \widehat{G}$$ with
$j(\lambda) (t) := \lambda (T^t \mu)$ (compare (a) of the preceding
remark).

\begin{theorem}\label{thm: conseq cont eigen}
 Let $\mu \in \WAP (G)$ be given. Then, the measures $\mu$ and
$\mu_{\mathsf{s}}$ have the same autocorrelation $\gamma$, and pure
point diffraction. In fact, for any $\lambda\in\SSS (\mu)$ the
functions $c^n_\lambda : \XX(\mu) \longrightarrow \CC$ defined by
$$c^n_\lambda (\omega) :=\int_{A_n} j(\lambda) (t) d\omega (t)$$
converge uniformly on  $\XX(\mu)$ to the Fourier Bohr coefficient
function $$c_{j(\lambda)}: \XX(\mu)\longrightarrow \CC, \omega \mapsto
c_{j(\lambda)}(\omega);$$
 $A_\lambda = |c_{j(\lambda)} (\omega)|^2$
does not depend on $\omega\in\XX(\mu)$ and  the diffraction is given
by
$$\widehat{\gamma} = \sum_{\lambda\in\SSS(\mu)} A_\lambda
\delta_{j(\lambda)}.$$
\end{theorem}
\begin{remark} Validity of uniform convergence of the  $|c^n_\lambda|^2$ to
the point parts of $\widehat{\gamma}$ is sometimes discussed under
the name of Bombieri-Taylor conjecture, see \cite{Len} for details.
Thus, our result gives in particular validity of Bombieri/Taylor
conjecture for the dynamical systems coming from weakly almost
periodic measures.
\end{remark}

\begin{proof}
We first show that the autocorrelations agree:  \ As $\XX(\mu),
\XX(\mu_s)$ are uniquely ergodic, each of $\mu$ and $\mu_s$ have
unique autocorrelation. Let us denote by $\gamma_1, \gamma_2$ the
two autocorrelations. Then, as $\XX(\mu)$ is uniquely ergodic, we
have, see e.g. \cite{BL},
\[
c*\widetilde{c}*\gamma_1(t) = < \phi_c, T_t \phi_c>
\]
with the inner product calculated in $L^2(\XX(\mu), \theta_\mu)$,
where $\phi_c $ is defined via $\phi_c (\nu) = \nu \ast c (0)$ see
above.

Then, as $\XX(\mu_{\mathsf{s}})$ is uniquely ergodic, we have
\[
c*\widetilde{c}*\gamma_2(t) = < \phi_c, T_t \phi_c>
\]
with the inner product calculated in $L^2(\SSS(\mu), \theta_\mu)$.
By Theorem \ref{more on dynamical hulls} those are equal, and hence
$$
c*\widetilde{c}*\gamma_1 = c*\widetilde{c}*\gamma_2$$ for all $ c
\in \Cc(G)$. This shows that
\[
\gamma_1=\gamma_2 \,.
\]

We now turn to proving pure point diffraction:  This follows
directly as  the systems have pure point dynamical spectra. In fact,
it is known that pure point diffraction and pure point dynamical
spectrum are equivalent for translation bounded measure dynamical
systems \cite{BL} (and in fact even more general situations
\cite{LS,LM}).

We finally discuss the formula for $\widehat{\gamma}$: By the
preceding corollary, the dynamical system $(\XX(\mu),G)$ has pure
point spectrum with continuous eigenfunctions to the eigenvalues
$j(\lambda)$, $\lambda\in\SSS(\mu)$. Thus,  the uniform convergence
of the $c^n_\lambda$ to a function $\widetilde{c}_\lambda$ follows
from Corollary 2 of \cite{Len}. Theorem 4 of \cite{Len} then gives
that $A_\lambda:= |\widetilde{c}_\lambda (\omega)|^2$ is independent
of $\omega \in\XX(\mu)$ and satisfies
$\widehat{\gamma}(\{j(\lambda)\}) = A_\lambda$. Now, the equality
$$\widehat{\gamma} =\sum_{\lambda\in
\SSS(\mu)} A_\lambda \delta_{j(\lambda)}$$ follows from Corollary 2
of \cite{Len}. It remains to show that $\widetilde{c}_\lambda
(\omega)$ equals the Fourier coefficient $c_{j(\lambda)}(\omega)$.
This, however, is clear from the definition of the Fourier
coefficients.
\end{proof}

\section{On the hull of the autocorrelation}\label{sec:hull-autororrelation} In the preceding considerations we
essentially always started with a weakly almost periodic measure.
Here, we discuss that \textit{any} measure dynamical system gives
rise to the hull of a weakly almost periodic measure viz its
autocorrelation and how  this hull carries important information on
the  pure point diffraction spectrum of the original system. In this
sense hulls of weakly almost periodic measures are somehow
unavoidable when dealing with diffraction.

\medskip

Let $(\XX,T)$ be a measure dynamical system and $m$ an invariant
probability measure on $\XX$. Then, there exists a unique positive
definite measure $\gamma$ with
$$\gamma\ast c\ast \widetilde{d} (0)  = \langle \phi_c ,\phi_d\rangle$$
for all $c,d\in \Cc (G)$, \cite{BL}. This measure is called the
autocorrelation of $m$.

Collecting the results from this paper we get:

\begin{theorem}\label{thm: mu to gamma}  Let $(\XX, G, m)$ be any ergodic system of translation bounded measures, and let $\gamma$ be the autocorrelation of this system. Then
\begin{itemize}
  \item [(a)] $\XX(\gamma)$ is a weakly almost periodic system, therefore, uniquely ergodic.
  \item [(b)] $\XX(\gamma)$ has pure point dynamical spectrum.
  \item [(c)] The   diffraction spectrum of   $(\XX,T,m)$ agrees with the
diffraction spectrum of $(\XX(\gamma),T)$ and generates the whole
dynamical spectrum of $(\XX(\gamma),T)$ as a group.
  \item [(d)] If $\XX =\XX(\mu)\subset \WAP(G)$ then
  \begin{enumerate}
  \item[(1)] $\XX(\gamma)$ is a compact group;
  \item[(2)] For $m$-almost all $\omega \in \XX$, $\XX(\gamma)$ is isomorphic as compact group with $\XX(\omega_{\mathsf{s}})$.
  \item[(3)] Composing $P_{\mathsf{s}} : \XX \to \SAP(G)$ with the isomorphism from $(2)$ gives us for $m$-almost all $\omega \in \XX$ a continuous factor $G$-mapping $\XX(\omega) \to \XX(\gamma)$.
  \end{enumerate}
\end{itemize}
\end{theorem}
\begin{proof}
(a) Since $\gamma$ is positive definite and translation bounded, it is weakly almost periodic \cite{MoSt}.

(b) Follows from (a) and Theorem \ref{cor:cont eigenfunctions}.

(c) By Theorem \ref{thm: conseq cont eigen}, the intensity of the
diffraction of $(\XX(\gamma), G)$ is given by
\[
A_\chi = \left| \widehat{\gamma}(\{ \chi \}) \right|^2 \,.
\]
It follows that $A_\chi \neq 0 \Leftrightarrow \widehat{\gamma}(\{
\chi \}) \neq 0$, that is the systems $(\XX(\gamma), G)$ and $(\XX,
m, G)$ have the same Bragg peaks.

Now, the system $(\XX(\gamma), G)$ is weakly almost periodic, and
hence it has pure point dynamical spectrum, which is the same as the
group generated by the  diffraction spectrum, \cite{BL}.

(d): (1) Since $\XX(\mu)$ is a weakly almost periodic dynamical system, and hence it has pure point diffraction. Then $\gamma \in \SAP(G)$ \cite{ARMA,MoSt}.

(2) Since $\gamma \in \SAP(G)$ the hull $\XX(\gamma)$ is a compact group whose dual is exactly the pure point dynamical spectrum of $\XX(\gamma)$.

By \cite{BL} for almost all $\omega \in \XX$, $\gamma$ is the autocorrelation of $\omega$. Therefore, by (c), the diffraction and dynamical spectra of $\omega$ are exactly the dual of $\XX(\gamma)$.

As $\XX \subset \WAP(G)$, it follows that for $m$-almost all $\omega$, the system$\XX(\omega)$ has also pure point spectrum, and the spectrum is the dual of the compact group $\XX(\omega_{\mathsf{s}})$.

Therefore $\XX(\omega_{\mathsf{s}})$ and $\XX(\gamma)$ are compact groups with isomorphic duals.

(3) Follows immediately from (2) and Lemma \ref{continuity of projections on hull}.
\end{proof}

\smallskip

Next, consider the set $\cE$ of all  dynamical systems of
translation bounded measures on $G$ equipped with invariant
probability measures, that is
\[
\cE := \{ (\XX, m, G) | \XX \subset \cM^\infty(G),  \mbox{ $m$
invariant probability measure on $\XX$}  \} \,.
\]
We can define a function $F : \cE \to \cE$ via
\[
F (\XX, m, G) = (\XX(\gamma), m', G) \,,
\]

where $\gamma$ is the autocorrelation of $(\XX, m, G)$ and $m'$ is
the unique ergodic measure on $\XX(\gamma)$. Then, the previous
theorem yields the following consequence.

\begin{proposition}\label{Prop 5.2} Let $(\XX,m,G) \in \cE$, let $S$ denote the pure point spectrum of $(X,m,G)$ and let $\TT =\widehat{S}$.
\begin{itemize}
\item[(a)] $F(\XX,m,G)$ is a weakly almost periodic system with spectrum $S$.
\item[(b)] For each $n \geq 2$, $F^n((\XX,m,G))$ is an almost periodic system, which is isomorphic as an abelian group to $\TT$.
\end{itemize}
\end{proposition}
\begin{proof}

(a) Follows from Theorem~\ref{thm: mu to gamma} (a).

(b) Follows by induction from Theorem ~\ref{thm: mu to gamma} (c) and (d).
\end{proof}

\begin{remark} Note that while they are topologically isomorphic as dynamical systems, the systems $F^n( (\XX, m, G)))$ have typically different autocorrelation measures, and hence are in general supported on different subsets of $\cM^\infty(G)$.
\end{remark}

Finally, if we denote by $\gamma_{(0)}$ the autocorrelation of $(\XX,m,G)$ and by $\gamma_{(n)}$ the autocorrelation of $F^{n}(\XX,m,G)$ we get
\begin{proposition}
\begin{itemize}
\item[(a)] $\gamma_{(0)} \in \WAP(G)$.
\item[(b)] $\gamma_{(n)} \in \SAP(G)$ for all $n \geq 1$.
\item[(c)] For each $n \geq 1$ and all $\chi \in \widehat{G}$ we have
\[
c_\chi(\gamma_{(n+1)}) = \left| c_\chi(\gamma_{n}) \right|^2 = \left| c_\chi(\gamma_{(0)}) \right|^{2^{n+1}} \,.
\]
\end{itemize}
In particular, all $\gamma_{(n)}$ have the same set of characters
with non-trivial Fourier-Bohr coefficient.
\end{proposition}
\begin{proof}

(a) and (b) are immediate consequences of Proposition~\ref{Prop 5.2}.

(c) Since $\gamma_{n+1}$ is the autocorrelation of $\XX(\gamma_n)$, by Theorem \ref{thm: conseq cont eigen} we have
\[
\widehat{\gamma_{n+1}}(\{ \chi \}) = A_\chi = \left| c_\chi(\gamma_n) \right|^2 \,.
\]
Since $\gamma_{n+1}$ is Fourier transformable, we also have
\[
\widehat{\gamma_{n+1}}(\{ \chi \}) = c_\chi(\gamma_{n+1}) \,,
\]
which completes the proof.
\end{proof}

\begin{remark}
The strong almost periodicity of $\gamma_1$ is the key for the proof of \cite[Thm.~7.1]{NS11}.
\end{remark}

\section{Application to weighted Dirac combs}\label{sec:Application}
In this Section we apply part of the results of the previous
sections to study the support of the Eberlein decomposition for
weighted Dirac combs.

\smallskip

Recall that at subset $\varLambda $ of $G$  is called
\textbf{uniformly discrete} if there exists an open set $U$
containing the neutral element of $G$ such that $(x + U) \cap (y +
U) = \emptyset$ for all $x,y\in\varLambda$ with $x\neq y$. Such a
subset is called \textbf{relatively dense} if there exists a compact
set $C$ such that $\varLambda + C = G$. A set which is both
relatively dense and uniformly discrete is called a \textbf{Delone
set}.

A set $\varLambda \subset G$ is called a \textbf{Meyer set} if
$\varLambda$ is relatively dense and if
$\varLambda-\varLambda-\varLambda$ is uniformly discrete. If $G$ is
compactly generated, then the second condition can be replaced by
the weaker  $\varLambda-\varLambda$ is uniformly discrete
\cite{BLM,NS11}.

To connect subsets to the measures we note   that any uniformly
discrete set (and hence any Delone set) $\varLambda$ gives rise to a
measure
$$\delta_\varLambda := \sum_{x\in \varLambda} \delta_x,$$
where $\delta_x$ denotes the unique point mass at $x$. This measure
is known as the \textbf{Dirac comb of $\varLambda$.} The map
$$\delta : \mbox{Uniform discrete subset of $G$}\longrightarrow
\cM^\infty$$ is injective. In this way, the set of uniform discrete
subsets of $G$ inherits a topology and we can form also the
\textbf{hull} $\XX (\varLambda)$ of $\varLambda$. Clearly, this hull
is homeomorphic to $\XX (\delta_\varLambda)$ via $\delta$.

\subsection{On almost periodic Delone sets}\label{sec:Delone}
Here, we consider the case that all weights are equal to one i.e. we
look at some consequences of Theorem \ref{thm:Hull WAP} for Delone
sets.

\smallskip

\begin{corollary}\label{IP}
Let $\Lambda$ be a Delone set such that $\delta_\Lambda$ is a weakly
almost periodic measure. Then there exists some Delone set
$\Lambda'$ so that $(\delta_\Lambda)_{\mathsf{s}}=
\delta_{\Lambda'}$.
\end{corollary}

\smallskip

\begin{remark}
It doesn't  follow from Proposition \ref{IP} that $\Lambda$ itself
is strongly  almost periodic (in the sense that $\delta_\varLambda$
is strongly almost periodic). For example, if $\Lambda=\ZZ
\backslash\{0 \}$ then $\left( \delta_\Lambda \right)_{\mathsf{s}}=
\delta_\ZZ$ which is indeed an element of $\XX(\delta_\Lambda)$, but
$\Lambda$ is not strongly  almost periodic. The problem is that
while $\ZZ \in \XX(\Lambda)$, we don't have $\Lambda \in \XX(\ZZ)$.
\end{remark}

An interesting question is what happens if we also ask for $\Lambda$
to be minimal. We address this question next.

\smallskip

Let us recall first a Theorem of Favarov \cite{FAV} and
Kellendonk-Lenz \cite{KL}:

\begin{theorem}\label{FAVKL} Let $\Lambda$ be a Delone set with FLC. If $\delta_\Lambda$ is strongly almost periodic, then $\Lambda$ is a fully periodic crystal.
\end{theorem}

\begin{corollary} Let $\Lambda$ be a Delone set with FLC. If $\delta_\Lambda$ is weakly almost periodic and $\XX(\Lambda)$ is minimal, then $\Lambda$ is a fully periodic crystal.
\end{corollary}

Combining Theorem \ref{FAVKL} with Theorem \ref{thm:Hull WAP} we
get:

\begin{theorem}\label{StrongCordoba}
Let $\Lambda$ be a Delone set with FLC such that $\delta_\Lambda$ is
weakly almost periodic. Then, there exists a lattice $L$ and a
finite set $F$ so that
\[
F+L \in \XX(\Lambda) \,;\, \mbox{ and }
\left(\delta_{\vL}\right)_{\mathsf{s}}=\delta_{L+F}=\delta_L*\delta_F
\,.
\]
\end{theorem}

\smallskip

Let us look next at two interesting examples of weakly almost
periodic measures. Note that these examples are coming from Delone
sets without FLC, and emphasize the importance of FLC in the
previous results.

\begin{example} Let
\[\Lambda := \{ n+\frac{1}{n} | n \in \ZZ \backslash \{0 \} \} \cup \{0 \}  \,.\]
Let $\nu:= \delta_\Lambda- \delta_\ZZ$. Then $\nu*g$ is a function
vanishing at infinity for all $g \in \Cc(\RR)$ \cite{NS1}, and thus
null-weakly almost periodic \cite{EBE}.  This proves that $\nu$ is a
null weakly almost periodic measure. Hence, using the uniqueness of
the almost periodic decomposition, we have:
\[(\delta_\Lambda)_{\mathsf{s}}=\delta_\ZZ \,;\, (\delta_\Lambda)_0=\nu \,.\]
\end{example}

\begin{example} Let
\[\Lambda:=  \{ (n+\frac{1}{n}, 2m) | n \in \ZZ \backslash \{0 \} \,;\, m \in \ZZ \} \cup (n \sqrt{2}, 2m+1) | m,n \in \ZZ \}\subset \RR\times\RR \,.\]
A similar computation as in the previous example shows that, with
$\Lambda'= [\ZZ \times 2\ZZ] \cup [\ZZ\sqrt{2} \times (2\ZZ+1)]$ we
have
\[(\delta_\Lambda)_{\mathsf{s}}= \delta_{\Lambda'} \,.\]

Note that $\Lambda'$ doesn't have FLC, and is the union of
translates of lattices.

\end{example}

\subsection{The support of the Eberlein decomposition}\label{sect: ebe decomp and supp} Given a weakly almost periodic
measure $\mu$, we know that $\mu_{\mathsf{s}} \in \XX(\mu)$. We use
this result to show that for weighted dirac combs the support of
$\mu_{\mathsf{s}}$ and $\mu_0$ cannot be much larger than the
support of $\mu$.  We then use this result to rederive some recent
results about Eberlein decomposition of weighted Dirac combs with
Meyer set support \cite{NS5}.

\bigskip

Let us start by recalling a result of \cite{JBA}.

\begin{proposition}\label{prop1 jba} \cite[Prop.~5.2]{JBA}
Let $S$ be a set with finite local complexity and
\[
\mu= \sum_{x \in S} \omega_x \delta_x \,,
\]
be a translation bounded measure supported inside $S$. If $\nu \in
\XX(\mu)$ then there exists some $S' \in \XX(S)$ such that $\nu$ is
supported inside $S'$.
\end{proposition}

Combining this result with Theorem \ref{thm:Hull WAP} we get

\begin{theorem}\label{wap flc decomp}
Let $\mu$ be a weakly almost periodic measure supported on a set
with FLC. Then $\mu_{\mathsf{s}}$ and $\mu_0$ are pure point
measures and $\sup(\mu_{\mathsf{s}})$ has FLC.
\end{theorem}
\begin{proof}
It follows from Proposition \ref{prop1 jba} that
$\sup(\mu_{\mathsf{s}})$ has FLC. In particular $\mu_{\mathsf{s}}$
is a pure point measure. Since both $\mu$ and $\mu_{\mathsf{s}}$ are
pure point measures, so is their difference $\mu_0$.
\end{proof}

As consequence we get the following result about diffraction.

\begin{theorem}
Let $\mu$ be a translation bounded measure and let $\gamma$ be an
autocorrelation of $\omega$. If $\gamma$ is supported on an FLC set
then the following statements hold:
\begin{itemize}
\item[(a)] $\widehat{\gamma}_{d} \in \SAP(\widehat{G})$ and $\widehat{\gamma}_c \in \SAP(\widehat{G})$.
\item[(b)] Each of the pure point and continuous spectra is either empty or relatively dense.
\end{itemize}
\end{theorem}
\begin{proof} (a) By Theorem \ref{wap flc decomp}, each of $\mu_{\mathsf{s}}$ and
$\mu_0$  is a pure point measure. Moreover, since $\gamma$ is
positive definite, it is Fourier transformable, and hence so are
$\mu_{\mathsf{s}}$ and $\mu_0$ \cite{MoSt}.  Now since the Fourier
transform of a pure point Fourier transformable measure is a
strongly almost periodic measure \cite{ARMA,MoSt}, (a) follows.

(b) follows now immediately by combining (a) with
\cite[Prop.~4.5]{NS1}.
\end{proof}

In particular we get new proofs of some of the results in
\cite{NS5}.

\begin{theorem}\cite[Thm.~5.3]{NS5}
Let $\mu$ be a translation bounded measure supported inside a
Meyer set and let $\gamma$ be an autocorrelation of $\omega$. Then
\begin{itemize}
\item[(a)] $\widehat{\gamma}_{d} \in \SAP(\widehat{G})$ and $\widehat{\gamma}_c \in \SAP(\widehat{G})$.
\item[(b)] Each of the pure point and continuous spectra is either empty or relatively dense.
\end{itemize}
\end{theorem}

We complete the section by showing that in this case, if we consider
a cut and project scheme which gives our Meyer set, the
decomposition doesn't go outside a translate of the closure of the
window. This gives a result similar to \cite[Thm.~4.6]{NS5}.

Let us start by recalling that a \textbf{ Cut and project scheme} is
a triple $(G,H,\cL)$ consisting of two LCAG's $G$ respectively $H$
and a lattice $\cL \subset G \times H$ such that
\begin{itemize}
  \item {} the restriction of the first projection $\pi_G : G \times H \to H$ to $\cL$ is one to one.
  \item {} the second projection $\pi_H : G \times H \to H$ satisfies
$\pi_H(\cL)$ is dense in $H$.
\end{itemize}

Given any cut and project scheme $(G,H,\cL)$, each pre-compact set $W \subset H$ defines an uniformly discrete set
$\oplam(W) \subset G$ via
\[
\oplam(W) := \{ x \in G | \mbox{there exists $y \in W$  such that
$(x,y) \in \cL$} \}\,.
\]

Next we need to recall the following lemma.

\begin{lemma} \cite[Thm.~5.9(i)]{JBA} Let $(G,H, \cL)$ be a cut and project scheme, $W \subset G$ be compact and
\[
\mu:= \sum_{x \in \oplam(W)} \omega_x \delta_x \,,
\]
be a translation bounded measure. If $\nu \in \XX(\mu)$ then there
exists some $(s,t) \in G\times H$ such that
\[
\sup(\nu) \subset s+\oplam(t+W) \,.
\]
\end{lemma}

Combining this result with Theorem \ref{thm:Hull WAP} we get (compare with \cite{NS5})

\begin{theorem}\label{thm: new them supp meyer} Let $(G,H, \cL)$ be a cut and project scheme, $W \subset G$ be compact and
\[
\mu:= \sum_{x \in \oplam(W)} \omega_x \delta_x \,,
\]
be a weakly almost periodic measure. Then there exists some $(s,t)
\in G\times H$ such that
\[
\sup(\mu_{\mathsf{s}}) \subset s+\oplam(t+W) \, \mbox{ and } \,
\sup(\mu_{0}) \subset \left(s+\oplam(t+W)\right) \cup \oplam(W) \,.
\]
\end{theorem}

\begin{remark}
As the difference $\mu-\mu_{\mathsf{s}}$ is a null weakly almost
periodic measure with Meyer set support, it is essentially a small
measure, and hence their supports cannot differ "too much". This is
most likely the reason why in \cite{NS5} the author proved that
$(s,t)$ in Theorem \ref{thm: new them supp meyer} can be chosen to
be $(0,0)$.
\end{remark}

\section{Eberlein convolution of weakly almost periodic
measures}\label{sec:ebe-conv} In this section we have a thorough
look at  the Eberlein convolution of weakly almost periodic measures
measures.  In particular, we show that this convolution always
exists and is independent of the choice of van Hove sequence. We then
go on and prove that it defines a strongly almost periodic measure.
As an application we obtain an alternative proof for the pure
pointedness of the autocorrelation of a weakly almost periodic
measure (already shown above in Section \ref{sec:dynam spectrum}).

\bigskip

\begin{lemma}\label{ebe conv lemma} Let $\mu , \nu \in \cM^\infty(G)$. If the limit
\[\varpi = \lim_n \frac{1}{\theta_G (A_n)} (\mu|_{A_n})* (\nu|_{A_n}) \,, \]
exists then for all $f , g \in \Cc(G)$ we have
\[
\varpi*f*g (t) = \lim_n \frac{1}{\theta_G (A_n)} \int_{A_n} (\mu * f) (s) (\nu*g)(t-s) d s \,.
\]
\end{lemma}
\begin{proof}: Fix $t \in G$.

\begin{equation}\label{EQ1}
\begin{split}
& \left|  \frac{ \int_{G} (f*\mu)(s) (g*\nu)(t-s) 1_{A_n}(s)ds}{\theta_G(A_n)}- \frac{ \int_{G} (f*\mu_{A_n})(s) (g*\nu_{A_n})(t-s) ds}{\theta_G(A_n)} \right|\\
\leq&  \frac{ \int_{G} \left|  (f*\mu)(s)1_{A_n}(s) (g*\nu)(t-s) -  (f*\mu_{A_n})(s) (g*\nu_{A_n})(t-s) ds \right|}{\theta_G(A_n)} \,.
\end{split}
\end{equation}

Let $K_0$ be any compact set containing $\pm \supp(f)$ and let $K=K_0 \cup (t+ K_0)$. We claim that
\[
(f*\mu)(s)1_{A_n}(s) (g*\nu)(t-s) -  (f*\mu_{A_n})(s) (g*\nu_{A_n})(t-s) \neq 0 \Rightarrow s \in \partial^K(A_n) \,.
\]

 \emph{Case 1:} $s \in A_n$. Then
\[
(f*\mu)(s) (g*\nu)(t-s) \neq  (f*\mu_{A_n})(s) (g*\nu_{A_n})(t-s) \,.
\]

This implies $f*\mu(s) \neq f*\mu_{A_n} (s)$ or $g*\nu(t-s) \neq g*\nu_{A_n}(t-s)$.

\emph{Subcase 1.a:} If $f*\mu(s) \neq f*\mu_{A_n} (s)$  then
\[\int_G  f(s-r) (1-1_{A_n}(r)) d \mu(r) \neq 0 \,.\]

Therefore, there exists some $r$ such that $ f(s-r) (1-1_{A_n}(r)) \neq 0$. Then $r \notin A_n$ and $s-r \in \supp(f) \subset K$.

Thus $s \in (G \backslash A_n)+K$. Since $s \in A_n$ we get $s \in \partial^K(A_n) \,.$

\emph{Subcase 1.b:} $g*\nu(t-s) \neq g*\nu_{A_n}(t-s)$ then
\[\int_G  g(t-s-r) (1-1_{A_n}(r)) d \mu(r) \neq 0 \,.\]

Therefore, there exists some $r$ such that $ g(t-s-r) (1-1_{A_n}(r)) \neq 0$. Then $r \notin A_n$ and $t-s-r \in \supp(f) \subset K_0$. Therefore
\[
s \in t-r -K_0 \subset t -K_0-r \subset K-r \,.
\]

As $r \notin A_n =-A_n$ it follows that $s \in (G \backslash A_n)+K$. This completes Case 1.

\emph{ Case 2:} $s \notin A_n$. Then, exactly as before
\[
(f*\mu)(s)1_{A_n}(s) (g*\nu)(t-s) \neq  (f*\mu_{A_n})(s) (g*\nu_{A_n})(t-s)
\]
implies
\[(f*\mu)(s)1_{A_n}(s)  \neq  (f*\mu_{A_n})(s) \, \mbox { or } \, 1_{A_n}(s) (g*\nu)(t-s) \neq  (g*\nu_{A_n})(t-s) \,.
\]

As $1_{A_n}(s)=0$ we get
\[
0 \neq  (f*\mu_{A_n})(s) (g*\nu_{A_n})(t-s)
\]
This implies that $\int_G f(s-t) d \mu_{A_n}(t) \neq 0$, and there fore, there exists some $t \in A_n$ so that $f(s-t) \neq 0$. Hence $s-t \in \supp(f)$ and therefore, $s \in A_n +K$.

It follows that $s \in (A_n+K) \backslash A_n$. This completes Case 2.

Now, since  $(f*\mu)(s)1_{A_n}(s) (g*\nu)(t-s) -  (f*\mu_{A_n})(s) (g*\nu_{A_n})(t-s)=0$ for $s$ outside $\partial^K(A_n)$, we get in (\ref{EQ1}):

\begin{eqnarray*}
&& \left|  \frac{ \int_{G} (f*\mu)(s) (g*\nu)(t-s) 1_{A_n}(s)ds}{\theta_G(A_n)}- \frac{ \int_{G} (f*\mu_{A_n})(s) (g*\nu_{A_n})(t-s) ds}{\theta_G(A_n)} \right|\\
&\leq&  \frac{ \int_{G} \left|  (f*\mu)(s)1_{A_n}(s) (g*\nu)(t-s) -  (f*\mu_{A_n})(s) (g*\nu_{A_n})(t-s) ds \right|}{\theta_G(A_n)} \\
&=& \frac{ \int_{\partial^K(A_n)} \left|  (f*\mu)(s)1_{A_n}(s) (g*\nu)(t-s) -  (f*\mu_{A_n})(s) (g*\nu_{A_n})(t-s) ds \right|}{\theta_G(A_n)} \\
&\leq& \frac{ \int_{\partial^K(A_n)}2 \| |f|*|\mu| \|_\infty \| |g|*|\nu| \|_\infty}{\theta_G(A_n)} \\
&=& 2 \| |f|*|\mu| \|_\infty \| |g|*|\nu| \|_\infty \frac{ \theta_G (\partial^K(A_n))}{\theta_G(A_n)} \,.
\end{eqnarray*}

Therefore, by the van Hove sequence property, we get that
\[
\lim_n \left|  \frac{ \int_{G} (f*\mu)(s) (g*\nu)(t-s) 1_{A_n}(s)ds}{\theta_G(A_n)}- \frac{ \int_{G} (f*\mu_{A_n})(s) (g*\nu_{A_n})(t-s) ds}{\theta_G(A_n)}  \right| =0 \,.
\]

As
\[
 \lim_n \frac{ \int_{G} (f*\mu_{A_n})(s) (g*\nu_{A_n})(t-s) ds}{\theta_G(A_n)} = \varpi*f*g (t) \,,
\]
it follows that
\begin{eqnarray*}
\lim_n  && \frac{ \int_{A_n} (f*\mu)(s) (g*\nu)(t-s) ds}{\theta_G(A_n)} \\
  =\lim_n   && \frac{ \int_{G} (f*\mu)(s) (g*\nu)(t-s) 1_{A_n}(s)ds}{\theta_G(A_n)} = \varpi*f*g (t) \,.
\end{eqnarray*}
\end{proof}

\smallskip

Using this result we can now prove the existence of the Eberlein convolution of weakly almost periodic measures.

\begin{theorem}\label{genEC} Let $\mu , \nu \in \WAP(G)$. Then the limit
\[\mu \circledast \nu := \lim_n \frac{1}{\theta_G (A_n)} (\mu|_{A_n})* (\nu|_{A_n}) \,, \]
exists, is translation bounded and is independent of the choice of the van Hove sequence.

Moreover, for all $f , g \in \Cc(G)$ we have $\left(\mu \circledast \nu \right)*f*g =(f * \mu) \circledast (g *\nu)$.
\end{theorem}
\begin{proof}

By \cite{BL}, there exists a constant $C$ and a compact set $K \subset G$ so that for all $n$ we have
\[
\frac{1}{\theta_G (A_n)} (\mu|_{A_n})* (\nu|_{A_n}) \in \cM^C_K(G) := \{ \varpi \in \cM^\infty(G) | \| \mu \|_K \leq C \} \,.
\]
Moreover, the set $\cM^C_K(G)$ is compact in the vague topology.

Thus, to prove that the limit exists, it suffices to show that any two cluster points of $\frac{1}{\theta_G (A_n)} (\mu|_{A_n})* (\nu|_{A_n})$ are equal. Let $\gamma_1, \gamma_2$ be any two cluster points of $\frac{1}{\theta_G (A_n)} (\mu|_{A_n})* (\nu|_{A_n})$.

Let $f, g \in \Cc(G)$. Then, the functions $f * \mu$ and $g *\nu$ are weakly almost periodic, therefore their Eberlein $(f * \mu) \circledast (g *\nu)$ convolution exists, and it can be calculated with respect to any van Hove sequence.

Let $A_{k_n}$ and $A_{l_n}$ be the subsequences which define these cluster points:
\begin{equation}\label{clust1}
\gamma_1=\lim_n \frac{1}{\theta_G (A_{k_n})} (\mu|_{A_{k_n}})* (\nu|_{A_{k_n}}) \,,
\end{equation}
and
\begin{equation}\label{clust2}
\gamma_2=\lim_n \frac{1}{\theta_G (A_{l_n})} (\mu|_{A_{l_n}})* (\nu|_{A_{l_n}}) \,,
\end{equation}

By applying Lemma \ref{ebe conv lemma} to (\ref{clust1}) we get:
\[
\gamma_1*f*g(s)=\lim_n \frac{1}{\theta_G (A_{k_n})} \int_{A_{k_n}} (f* \mu)(s-t) (g *\nu)t dt =(f * \mu) \circledast (g *\nu) (s)  \,,
\]
while by applying Lemma\ref{ebe conv lemma} to \ref{clust2} we get:
\[
\gamma_2*f*g=\lim_n \frac{1}{\theta_G (A_{l_n})}\int_{A_{l_n}} (f* \mu)(s-t) (g *\nu)t dt =(f * \mu) \circledast (g *\nu) (s) \,.
\]

This shows that $\gamma_1*f*g=\gamma_2*f*g$ for all $f,g \in \Cc(G)$. In particular, evaluating at $0$ we get $\gamma_1(f*g^\dagger)=\gamma_2(f*g^\dagger)$, for all $f, g \in \Cc(G)$.
It follows that $\gamma_1, \gamma_2$ are two continuous linear functionals on $\Cc(G)$, which are equal on the dense subset $\{ f*g |f,g \in \Cc(G) \}$, therefore $\gamma_1 =\gamma_2$.

This shows the existence of the limit. The independence of the van Hove sequence is done exactly the same way: If $A_n, B_n$ are van Hove sequences, and
\[
\gamma_1=\lim_n \frac{1}{\theta_G (A_n)} (\mu|_{A_n})* (\nu|_{A_{n}}) \,,
\]
and
\[
\gamma_2=\lim_n \frac{1}{\theta_G (A_n)} (\mu|_{A_n})* (\nu|_{A_{n}}) \,,
\]
then redoing the above computation, we get $\gamma_1*f*g=(f * \mu) \circledast (g *\nu) =\gamma_2*f*g$ for all $f , g \in \Cc(G)$. Therefore, exactly as before $\gamma_1=\gamma_2$.

The fact that $\mu \circledast \nu$ is translation bounded follows immediately from the fact that it is a vague cluster point of a net from $\cM^C_K(G)$, which is compact in the vague topology.

The last claim is obvious.
\end{proof}

\begin{definition} Let $\mu , \nu \in \WAP(G)$. We call the limit
\[\mu \circledast \nu := \lim_n \frac{1}{\theta_G (A_n)} (\mu|_{A_n})* (\nu|_{A_n}) \,, \]
the {\bf Eberlein convolution} of $\mu$ and $\nu$.
\end{definition}

Exactly as in the case of autocorrelation, the Eberlein convolution can be calculated by truncating only one term:

\begin{lemma} Let $\mu , \nu \in \WAP(G)$. Then $\frac{1}{\theta_G (A_n)} (\mu|_{A_n})* \nu$ and $\frac{1}{\theta_G (A_n)} \mu*  (\nu|_{A_n})$ converge in the vague topology to $\mu \circledast \nu$.
\end{lemma}
\begin{proof}
Let $f \in \Cc(G)$. Then
\begin{eqnarray*}
&& \left( \frac{1}{\theta_G (A_n)} (\mu|_{A_n})* \nu -\frac{1}{\theta_G (A_n)} (\mu|_{A_n})* (\nu|_{A_n}) \right)(f) =  \\
&=&\frac{1}{\theta_G (A_n)} \int_G \int_G f(s+t) d \mu|_{A_n}(s) d (\nu-\nu|_{A_n})(t)   \\
&=&\frac{1}{\theta_G (A_n)} \int_G \int_G f(s+t) 1_{A_n}(s) (1-1_{A_n})(t) d \mu(s) d \nu(t)   \,. \\
\end{eqnarray*}

Then,
\begin{eqnarray*}
&& \left| \left(  \frac{1}{\theta_G (A_n)} (\mu|_{A_n})* \nu -\frac{1}{\theta_G (A_n)} (\mu|_{A_n})* (\nu|_{A_n}) \right)(f) \right| \leq  \\
&\leq&\frac{1}{\theta_G (A_n)} \int_G \int_G \left| f(s+t) \right| 1_{A_n}(s) \left| (1-1_{A_n})\right| (t) d \left| \mu \right|(s)  d \left| \nu \right|(t)  \,. \\
\end{eqnarray*}

Let us note that $\left| f(s+t) \right| 1_{A_n}(s) \left| (1-1_{A_n})\right| (t) \neq 0$ when $s \in A_n, t \in G \backslash A_n$ and $s+t \in \supp(f)$. Denoting $K=\pm \supp(f)$, and using $A_n=-A_n$ we get that
\[
\left| f(s+t) \right| 1_{A_n}(s) \left| (1-1_{A_n})\right| (t) \neq 0 \Rightarrow t \in \partial^K(A_n) \,.
\]

Therefore
\begin{eqnarray*}
&& \left| \left(  \frac{1}{\theta_G (A_n)} (\mu|_{A_n})* \nu -\frac{1}{\theta_G (A_n)} (\mu|_{A_n})* (\nu|_{A_n}) \right)(f) \right| \leq  \\
&\leq&\frac{1}{\theta_G (A_n)} \int_{\partial^K(A_n)} \int_G \left| f(s+t) \right| 1_{A_n}(s)  d \left| \mu \right|(s)  d \left| \nu \right|(t)   \\
&\leq&\frac{1}{\theta_G (A_n)} \int_{\partial^K(A_n)} \int_G \left| f(s+t) \right| d \left| \mu \right|(s)  d \left| \nu \right|(t)   \\
&\leq&\frac{1}{\theta_G (A_n)} \int_{\partial^K(A_n)} \left| f^\dagger \right| * \left| \mu \right| (t)  d \left| \nu \right|(t)   \\
&\leq& \| \left| f^\dagger \right| * \left| \mu \right| \|_\infty  \frac{1}{\theta_G (A_n)} \left| \nu \right|(\partial^K(A_n))  \,. \\
\end{eqnarray*}

Now, by using the translation boundedness of $\mu, \nu$ and the van Hove property, we get that
\[
\lim_n \left( \frac{1}{\theta_G (A_n)} (\mu|_{A_n})* \nu -\frac{1}{\theta_G (A_n)} (\mu|_{A_n})* (\nu|_{A_n}) \right)(f) =0 \,.
\]

This shows that in the vague topology we have
\[
\lim_n  \frac{1}{\theta_G (A_n)} (\mu|_{A_n})* \nu -\frac{1}{\theta_G (A_n)} (\mu|_{A_n})* (\nu|_{A_n})  =0 \,.
\]

Since
\[
\lim_n \frac{1}{\theta_G (A_n)} (\mu|_{A_n})* (\nu|_{A_n})  = \mu \circledast \nu \,.
\]
we get
\[
\lim_n  \frac{1}{\theta_G (A_n)} (\mu|_{A_n})* \nu  =\mu \circledast \nu \,.
\]

This completes the first claim. The second claim follows by interchanging $\mu$ and $\nu$.

\end{proof}

Note that Theorem \ref{genEC} tells us that the limit always exists,
and the definition doesn't depend on the choice of the van Hove
sequence.

\begin{theorem} Let $\mu , \nu \in \WAP(G)$ and let $A_n$ be a van Hove sequence. Then
\begin{itemize}\itemsep=2pt
  \item[(a)] $\mu \circledast \nu \in \SAP(G)$.
  \item[(b)]
  \[c_\chi(\mu \circledast \nu) = c_\chi (\mu) c_\chi(\nu) \]
    \item[(c)]
  \[(\mu \circledast \nu)_{\mathsf{b}} = (\mu)_{\mathsf{b}} *(\nu)_{\mathsf{b}} \]
\end{itemize}
\end{theorem}
\begin{proof}

(a) We know that for all $f , g \in \Cc(G)$ we have $\left(\mu \circledast \nu \right)*f*g =(f * \mu) \circledast (g *\nu)$. As the Eberlein convolution of two weakly almost periodic functions is a strong almost periodic function \cite{EBE,MoSt}, it follows that $\left(\mu \circledast \nu \right)*f*g \in SAP(G)$ for all $f,g \in \Cc(G)$.

But this implies \cite{MoSt} that $\mu \circledast \nu \in \SAP(G)$.

(b) For all $f,g \in \Cc(G)$ we have
\[
c_\chi(\left(\mu \circledast \nu \right)*f*g)= c_\chi( (f * \mu) \circledast (g *\nu)) \,.
\]

Basic properties of Fourier Bohr coefficients yield
\[
c_\chi(\left(\mu \circledast \nu \right)*f*g)= c_\chi( \mu \circledast \nu ) \left(\widehat{f*g}(\chi) \right)= c_\chi( \mu \circledast \nu ) \left(\widehat{f}(\chi) \widehat{g}(\chi) \right) \,.
\]
and
\[
c_\chi( (f * \mu) \circledast (g *\nu))= c_\chi( (f * \mu)) c_\chi (g *\nu))=c_\chi(\mu) \widehat{f}(\chi) c_\chi(\nu) \widehat{g}(\chi)
\]

Therefore

\[c_\chi(\mu \circledast \nu)\widehat{f}(\chi) \widehat{g}(\chi)  = c_\chi (\mu) c_\chi(\nu)\widehat{f}(\chi) \widehat{g}(\chi)  \]

As this is true for all $f,g \in \Cc(G)$, the claim follows now by picking some $f,g$ whose Fourier transform doesn't vanish at $\chi$.

(c) As both $(\mu \circledast \nu)_{\mathsf{b}}$ and $(\mu)_{\mathsf{b}} *(\nu)_{\mathsf{b}}$ are finite measures on the compact group $G_{\mathsf{b}}$, it suffices to prove that they have the same Fourier transform. Indeed
\[
\widehat{(\mu \circledast \nu)_{\mathsf{b}}}(\chi)= c_\chi(\mu \circledast \nu)= c_\chi(\mu) c_\chi( \nu) = \widehat{\mu_{\mathsf{b}}}(\chi)\widehat{\nu_{\mathsf{b}}}(\chi)= \widehat{(\mu)_{\mathsf{b}}* (\nu)_{\mathsf{b}}}(\chi) \,.
\]
\end{proof}

An immediate consequence of this is the following.

\begin{theorem}\label{them:ppd} Let $\mu \in \WAP(G)$. Then
\begin{itemize}\itemsep=2pt
  \item[(a)] $\mu$ has unique autocorrelation $\gamma= \mu \circledast \widetilde{\mu}$.
  \item[(b)] $\mu$ is pure point diffractive.
    \item[(c)] The intensity of the Bragg peaks is given by
   \[ \widehat{\gamma}(\{ \chi \}) = \left|c_\chi(\mu) \right|^2 \,.
   \]
\end{itemize}
\end{theorem}

\smallskip
\begin{remark} If $\chi$ is a Bragg peak, then by Theorem \ref{thm:bragg peaks have continuous eigenfunctions}, $c_\chi$ is a continuous eigenfunction of the system.

We also know that $\XX(\mu)$ is uniquely ergodic and has pure point dynamical spectrum. Then, the spectral group is generated by the set of Bragg peaks \cite{BL,LMS-1}. Therefore, each eigenvalue $\lambda$ can be written in the form $\lambda=\chi_1+..+\chi_k-\chi_{k+1}-..-\chi_n$ for some Bragg peaks $ \chi_1,...,\chi_n$.

Then, it follows that
\[
f_\lambda=c_{\chi_1} \cdot ...\cdot c_{\chi_{k}} \cdot \overline{c_{\chi_{k+1}} \cdot ...\cdot c_{\chi_{n}}} \,,
\]
is a continuous eigenfunction for the eigenvalue $\lambda$.

This yields and alternate proof of Corollary \ref{cor:cont eigenfunctions}.

\end{remark}

As immediate consequences of Theorem \ref{them:ppd} and Corollary \ref{cor:wap are ue} we get:

\begin{corollary} Let $\mu \in \cM^\infty(G)$ be a positive definite measure. Then $\mu$ is pure point diffractive and uniquely ergodic.
\end{corollary}

\begin{corollary} Let $\mu \in \WAP(G)$. Then $\mu$ has pure point dynamical spectra.
\end{corollary}

\begin{corollary}
\rm Let $\mu \in \WAP(G)$. Then $\mu$ and $\mu_{\mathsf{s}}$ have
the same diffraction.
\end{corollary}

We complete the paper by stating and proving a result which is hidden behind the proof of uniqueness of autocorrelation.

\begin{lemma}\label{L35} \rm Let $\mu, \nu$ be two translation bounded measures. If $\mu-\nu$ is null weakly almost periodic, then $\mu$ and $\nu$ have the same autocorrelation(s).
\end{lemma}

\begin{proof} \rm Let $\phi, \psi \in \Cc(G)$, and let $A_n$ be any van hove Sequence. Then

\begin{equation}
\begin{split}
&\left| \phi*(\mu|_{A_n})*\widetilde{(\mu|_{A_n})}* \psi- \phi*(\nu|_{A_n})*\widetilde{(\nu|_{A_n})}* \psi \right| \\
&\leq\left| \phi*(\mu|_{A_n})*\widetilde{(\mu|_{A_n})}* \psi-\phi*(\mu|_{A_n})*\widetilde{(\nu|_{A_n})}* \psi \right| \\
&+\left|\phi*(\mu|_{A_n})*\widetilde{(\nu|_{A_n})}* \psi- \phi*(\nu|_{A_n})*\widetilde{(\nu|_{A_n})}* \psi \right| \\
&\leq \left| \phi*(\mu|_{A_n})*\left[\widetilde{(\mu|_{A_n})-(\nu|_{A_n})}\right]* \psi \right| \\
&+ \left|\phi*\left[\mu|_{A_n}-\nu|_{A_n}\right]*\widetilde{(\nu|_{A_n})}* \psi \right| \,.
\end{split}
\label{EQL35}
\end{equation}

Lets now fix an $\epsilon >0$.

Since $\mu-\nu$ is null weakly almost periodic, we have
\[\lim_{n} \frac{ \int_{A_n} \left| (\widetilde{(\mu|_{A_n})-(\nu|_{A_n})})* \psi \right|(t) dt}{\theta_G(A_n)} =0 \,,\]
and
\[\lim_{n} \frac{ \int_{A_n} \left| \phi*(\mu|_{A_n}-\nu|_{A_n}) \right|(t) dt}{\theta_G(A_n)} =0 \,.\]

Let
\[C:= \max\{ \| \left| \phi \right| *\left| \mu \right| \|_\infty ; \| \left| \phi \right| *\left| \nu \right| \|_\infty ; \| \left| \psi \right| *\left| \mu \right| \|_\infty ; \| \left| \psi \right| *\left| \nu \right| \|_\infty \} \,.\]

Let $K$ be any compact set which contains $\pm \supp(\phi)$ and $\pm \supp(\psi)$.

An easy computation shows that
\[
\left| \phi*(\mu|_{A_n})*\left[\widetilde{(\mu|_{A_n})-(\nu|_{A_n})}\right]* \psi \right| \leq   C \cdot 1_{A_n+K} * \left| \left[\widetilde{(\mu|_{A_n})-(\nu|_{A_n})}\right]* \psi \right| \,,
\]
and
\[
\left|\phi*\left[\mu|_{A_n}-\nu|_{A_n}\right]*\widetilde{(\nu|_{A_n})}* \psi \right|  \leq   C \cdot 1_{A_n+K} * \left| \phi* \left[(\mu|_{A_n})-(\nu|_{A_n})\right]\right|\,.
\]
Then for all $t \in G$ we have
\begin{eqnarray*}
&& \frac{ \left| \phi*(\mu|_{A_n})*\left[\widetilde{(\mu|_{A_n})-(\nu|_{A_n})}\right]* \psi \right|(t) }{\theta_G(A_n)} \\
&\leq&  C  \frac{ 1_{A_n+K} * \left| \left[\widetilde{(\mu|_{A_n})-(\nu|_{A_n})}\right]* \psi \right| (t)  }{\theta_G(A_n)} \\
&=& C \frac{ \int_{A_n+K}  \left| \left[\widetilde{(\mu|_{A_n})-(\nu|_{A_n})}\right]* \psi \right| (t-s) ds  }{\theta_G(A_n)}\\
&\leq& C \frac{ \int_{A_n}  \left| \left[\widetilde{(\mu|_{A_n})-(\nu|_{A_n})}\right]* \psi \right| (t-s) ds  }{\theta_G(A_n)} \\
&+&C \frac{ \int_{\partial^K A_n}  \left| \left[\widetilde{(\mu|_{A_n})-(\nu|_{A_n})}\right]* \psi \right| (t-s) ds  }{\theta_G(A_n)} \,.
\end{eqnarray*}
Since $\mu-\nu$ is null weakly almost periodic, the function $\left[\widetilde{(\mu)-(\nu)}\right]* \psi$ is null weakly almost periodic and thus

\[\lim_n  \frac{ \int_{A_n}  \left| \left[\widetilde{(\mu)-(\nu)}\right]* \psi \right| (t-s) ds  }{\theta_G(A_n)} = 0 \,,\]
uniformly in $t$.

A simple computation shows that this implies that
\[\lim_n  \frac{ \int_{A_n}  \left| \left[\widetilde{(\mu|_{A_n})-(\nu|_{A_n})}\right]* \psi \right| (t-s) ds  }{\theta_G(A_n)} = 0 \,.\]

Also, since $\left[\widetilde{(\mu|_{A_n})-(\nu|_{A_n})}\right]* \psi$ is bounded, by the van Hove property we have
\[\lim_{n} C \frac{ \int_{\partial^K A_n}  \left| \left[\widetilde{(\mu|_{A_n})-(\nu|_{A_n})}\right]* \psi \right| (t-s) ds  }{\theta_G(A_n)} =0 \,.\]

This shows that
\[\lim_{n}  \frac{ \left| \phi*(\mu|_{A_n})*\left[\widetilde{(\mu|_{A_n})-(\nu|_{A_n})}\right]* \psi \right|(t) }{\theta_G(A_n)} =0 \,,\]
pointwise for $t \in G$.

Similarly we can show that
\[\lim_{n}  \frac{ \left|\phi*\left[\mu|_{A_n}-\nu|_{A_n}\right]*\widetilde{(\nu|_{A_n})}* \psi \right|(t) }{\theta_G(A_n)} =0 \,,\]
pointwise for $t \in G$.

Using this in (\ref{EQL35}) we get that for all $t \in G$ and all van Hove sequences $\{ A_n \}$  we have
\[\lim_{n}  \frac{ \left| \phi*(\mu|_{A_n})*\widetilde{(\mu|_{A_n})}* \psi- \phi*(\nu|_{A_n})*\widetilde{(\nu|_{A_n})}* \psi \right| (t) }{\theta_G(A_n)} =0 \,.\]

In particular, in the vague topology, we have
\[\lim_{n}  \frac{ (\mu|_{A_n})*\widetilde{(\mu|_{A_n})}- (\nu|_{A_n})*\widetilde{(\nu|_{A_n})} }{\theta_G(A_n)} =0 \,.\]

Thus the limit $\lim_{n}  \frac{ (\mu|_{A_n})*\widetilde{(\mu|_{A_n})} }{\theta_G(A_n)}$ exists if and only if the limit $\lim_{n}  \frac{(\nu|_{A_n})*\widetilde{(\nu|_{A_n})} }{\theta_G(A_n)}$ exists, and in this case they are the same.
 \end{proof}

\begin{remark} \rm In Lemma \ref{L35} the measures $\mu, \nu$ need not be weakly almost periodic.
\end{remark}

{\large \sc \bf Acknowledgment:} {\small Part of this work is
inspired by stimulating talks given at the 2016 meeting 'Dynamical
systems for aperiodicity' in Lyon. DL would like to thank the
organizers for inviting him there. NS was supported by NSERC, under
grant 2014-03762, and would like to thank NSERC for their support.
Part of this work was done when NS visited DL at Jena University,
and NS would like to thank the mathematics department for
hospitality. }

\appendix
\renewcommand{\theequation}{A\arabic{equation}}
\setcounter{equation}{0}


\begin{thebibliography}{99}
\small

\bibitem{AG} E.~Akin, E. ~Glasner, \textit{WAP Systems and Labeled
Subshifts}, preprint \texttt{arXiv:1410.4753}.


\bibitem{ARMA1}
L.~N ~Argabright, J. ~Gil ~de ~Lamadrid, \textit{ Fourier analysis of
unbounded measures on locally compact abelian groups}, Memoirs of
the Amer. Math. Soc., Vol \textbf{145}, 1974.

\bibitem{Aus} J. Auslander,
\emph{Minimal Flows and their Extensions}, North-Holland
Mathematical Studies 153, Elsevier 1988.



\bibitem{JBA}
J.-B. Aujogue, \textit{Pure Point/Continuous Decomposition of
Translation-Bounded Measures and Diffraction}, preprint
\texttt{arXiv:1510.06381}

\bibitem{ABKL} J.-B. Aujogue, M. Barge, J. Kellendonk, D. Lenz,
\textit{Equicontinuous factors, Proximality and Ellis semigroup for
Delone sets} in: \cite{KLS},




\bibitem{TAO}
M.~Baake, U.~Grimm, \textit{Aperiodic Order. Vol.~1: A Mathematical
Invitation} (Cambridge University Press, Cambridge), 2013.



\bibitem{BL}
M.~Baake, D.~Lenz, \textit{Dynamical systems on translation bounded
measures:\ Pure point dynamical and diffraction spectra}, Ergod.\
Th.\ \& Dynam.\ Syst. \textbf{24}, 1867--1893, 2004.
\texttt{arXiv:math.DS/0302231}.

\bibitem{BL3}
M.\ Baake,  D.\ Lenz, \textit{ Spectral notions of aperiodic order},
preprint, 2016. \texttt{arXiv:1601.06629}.

\bibitem{BLM}
M. ~Baake, D. Lenz, R.V. Moody, \textit{A characterisation of Model
sets via Dynamical systems}, Ergod. Th. \& Dynam. Syst. \textbf{27},
341--382, 2007. \texttt{arXiv:math.DS/0511648}.

\bibitem{BM}
M. Baake, R.V. Moody, \textit{Weighted Dirac combs with pure point diffraction},
J.\ reine angew.\ Math.\ (Crelle) \textbf{573}, 61--94, 2004.
\texttt{arXiv:math.MG/0203030}.


\bibitem{BLPS} S. Beckus, D. Lenz, F. Pogorzelski, M. Schmidt,
\textit{Diffraction theory for processes of tempered distributions},
in preparation.

\bibitem{BF}
C. ~Berg, G. ~Forst, \textit{ Potential Theory on Locally Compact
Abelian Groups}, Springer Berlin, New York, 1975.

\bibitem{DM}
X.~Deng, R.V.~ Moody, \textit{Dworkin's argument revisited: point
processes, dynamics, diffraction, and correlations}, J.\ Geom.\
Phys.\ {\bf 58} 506--541, 2008. \texttt{arXiv:0712.3287}.

\bibitem{DG}
T.~Downarowicz, E.~Glasner, \textit{Isomorphic extensions and
applications}, preprint, \texttt{arXiv:1502.06999}.


\bibitem{Dwo}
S.~Dworkin, \textit{Spectral theory and $X$-ray diffraction}, J.\
Math.\ Phys.\ {\bf 34} 2965--2967, 1993.

\bibitem{EBE}
W.F. ~Eberlein, \textit{Abstract ergodic theorems and weak almost
periodic functions}, Trans. Amer. Math. Soc., \textbf{67}, 217-24,
1949.

\bibitem{EBE2}
W.F. ~Eberlein, \textit{  A Note on Fourier-Stiltjes Transforms},
 Proc. Amer. Math. Soc, Vol  \textbf{6}, No. 2, 310-312, 1955.

 \bibitem{EBE3}
W.F. ~Eberlein, \textit{  The Point Spectrum of Weakly Almost
Periodic Functions},  Michigan Math J. \textbf{3}, 137-139,
1955-1956.

\bibitem{EN}
R.~Ellis, M. Nerurkar \textit{Weakly Almost Periodic Flows}, Trans.
Amer. Math. Soc., \textbf{313}, vol 1,  103-119, 1989.

\bibitem{FAV}
S. Favorov, \textit{Bohr and Besicovitch almost periodic discrete
sets and quasicrystals}, Proc. Amer. Math. Soc. \textbf{140},
1761-1767, 2012. \texttt{arXiv:math.MG/1011.4036}.

\bibitem{FU}
H.~Furstenberg, \textit{Recurrence in Ergodic Theory and
Combinatorial Number Theory}, Princeton Univ.\ Press, Princeton, NJ, 1981.

\bibitem{GLAS}
E.~Glasner, \textit{ Ergodic Theory via Joinings},
Mathematical Surveys and Monographs, Vol. \textbf{101}; 2003.


\bibitem{Gouere-1}
J.-B. Gou\'{e}r\'{e},  \textit{Quasicrystals and almost
periodicity}, Comm. Math. Phys. \textbf{255},655--681, 2005.
\texttt{arXiv:math-ph/0212012}.


\bibitem{GB}
U.~Grimm, M.~Baake \textit{Homometric Point Sets and Inverse
Problems} Z. Krist. \textbf{223} , 777-781, 2008.
\texttt{arXiv:0808.0094}.

\bibitem{KL}
J. Kellendonk, D. Lenz, \textit{ Equicontinuous delone dynamical
systems}, Canadian Journal of Mathematics \textbf{65}, 149--170,
2013. \texttt{math.DS/1105.3855}.


\bibitem{KLS} J.
Kellendonk, D. Lenz, J. Savinien (eds), \textit{Mathematics of
aperiodic order},  Progress in Mathematics \textbf{309}, Birkhaeuser
2016.

\bibitem{LAG}
J.~C. ~Lagarias, \textit{ Mathematical quasicrystals and the problem
of diffraction}. In: Directions in Mathematical Quasicrystals
(eds. M. Baake and R.V Moody ), CRM Monograph Series, Vol \textbf{13},
Amer. Math. Soc., Providence, RI, 61-93, 2000.

\bibitem{ARMA}
J . ~Gil. ~de ~Lamadrid, L.~N ~Argabright, \textit{  Almost Periodic
Measures}, Memoirs of the Amer. Math. Soc., Vol \textbf{85}, No.
428, 1990.

\bibitem{LOO}
L.~H.~ Loomis, \textit{An introduction to abstract harmonic analysis}, D. Van Nostrand Company, Toronto-New York-London, 1953.


\bibitem{LMS-1}
J.-Y.\ Lee, R.~V.~Moody,  B.~Solomyak, \textit{Pure point dynamical
and diffraction spectra}, Annales H.\ Poincar\'{e} {\bf 3}
1003--1018; mp\_arc/02-39, 2002.


\bibitem{Len} D. Lenz, \textit{An autocorrelation and discrete
spectrum for dynamical systems on metric spaces}, preprint 2016.

\bibitem{LM}
D. Lenz, R.V. Moody, \textit{Stationary processes with pure point
diffraction},  to appear in:  Ergodic Theory \& Dynam.  Sys.,
\texttt{arXiv:1111.3617v1}.


\bibitem{LR}
D. Lenz, C. Richard, \textit{ Pure Point Diffraction and Cut and
Project Schemes for Measures: the Smooth Case}, Mathematische
Zeitschrift \textbf{256}, 347--378, 2007. \texttt{math.DS/0603453}.


\bibitem{LS}
D. Lenz, N. Strungaru, \textit{ Pure point spectrum for measurable
dynamical systems on locally compact Abelian grouls}, J. Math. Pures
Appl. \textbf{92} , 323--341, 2009. \texttt{arXiv:0704.2498}.

\bibitem{LS4}
D. Lenz, N. Strungaru, \textit{ On uniquely ergodicity of systems of measures}, in preparation.


\bibitem{MS} R. V. Moody, N. Strungaru,
\textit{Point sets and dynamical systems in the autocorrelation
topology}, Canad. Math. Bull. Vol. \textbf{47}, 82-99, 2004.



\bibitem{ReiterSte}
H.~Reiter, J. D. Stegeman \textit{ Clasical Harmonic Analysis and
Locally Compact Groups}, London Mathematical Society Monographs,
Clarendon Press, 2000.


\bibitem{CR}
C. Richard, \textit{ Dense Dirac combs in Euclidean space with pure point diffraction}, J.\ Math.\ Phys. \textbf{44}, 4436--4449, 2003.
\texttt{arXiv:math-ph/0302049}.

\bibitem{CRS2}
C. Richard, N. Strungaru, \textit{ A short guide to pure point diffraction in cut-and-project sets}, preprint.
\texttt{arXiv:1606.08831}.





\bibitem{MoSt}
R.V. Moody, N. Strungaru, \textit{Almost Periodic Measures and their Fourier
Transforms: Results from Harmonic Analysis}, to appear in:\
M.~Baake and U.~Grimm (eds.),
\textit{Aperiodic Order. Vol.\ $2$: Crystallography and
Almost Periodicity},
Cambridge Univ.\ Press, in preparation.

\bibitem{Ped}
G.~K.~Pedersen, \textit{Analysis Now}, Springer, New York (1989);
Revised \ printing (1995).

\bibitem{Rad} C. Radin,  \textit{Miles of tiles},
Ergodic theory of $Z^d$-actions (Warwick, 1993–1994),  237--258,
London Math. Soc. Lecture Note Ser., \textbf{228}, Cambridge Univ.
Press, Cambridge, 1996.

\bibitem{Sol} B.~Solomyak, \textit{Spectrum of dynamical systems arising from
Delone sets}, in: J.~Patera (ed.), \textit{Quasicrystals and
Discrete Geometry}, Fields Institute Monographs,  \textbf{10}, AMS,
Providence, RI (1998),  pp.\ 265--275.


\bibitem{NS1}
N.~Strungaru, \textit{Almost periodic measures and long-range
order in Meyer sets}, Discrete and Computational Geometry vol.
\textbf{33}(3), 483-505, 2005.

\bibitem{NS4}
N.~Strungaru, \textit{ Almost periodic Measures and Bragg Diffraction}, J. Phys. A: Math. Theor. \textbf{46}, 125205, 2013.
\texttt{arXiv:1209.2168}.

\bibitem{NS5}
N.~Strungaru, \textit{ On Weighted Dirac Combs Supported Inside Model Sets}, J. Phys. A: Math. Theor. \textbf{47},  2014.
\texttt{arXiv:1309.7947}.

\bibitem{NS11}
N.~Strungaru, \textit{Almost periodic measures and Meyer sets}, to appear in:\
M.~Baake and U.~Grimm (eds.),
\textit{Aperiodic Order. Vol.\ $2$: Crystallography and
Almost Periodicity},
Cambridge Univ.\ Press, in preparation.
\texttt{arXiv:1501.00945}.

\bibitem{ST}
N.~Strungaru,  V.~Terauds,  \textit{ Diffraction theory and almost
periodic distributions}, J. Stat. Phys. \textbf{164}, no. 5,
1183--1216, 2016. \texttt{arXiv:1603.04796}.

\bibitem{ter}
V.\ Terauds, \textit{The inverse problem for pure point diffraction -- examples and open questions},
J. Stat. Phys. \textbf{152}, no. 5, 954--968, 2013.
\texttt{arXiv:1303.3260}.

\bibitem{teba}
V. Terauds, M. Baake, \textit{Some comments on the inverse problem of pure point
diffraction}, in: \textit{Aperiodic Crystals},
eds.\ S.\ Schmid, R.L.\ Withers and R.\ Lifshitz,
Springer, Dordrecht , 35--41, 2013.
\texttt{arXiv:1210.3460}.



\bibitem{WA} P. Walters, \textit{ An introduction to ergodic theory},
Springer Graduate Texts in Mathematics, Springer-Verlag, New York,
1982.

\end{thebibliography}
\end{document}